\documentclass[11pt]{amsart}
\usepackage[left=2.5cm,right=2.5cm,top=3.5cm,bottom=3.5cm,a4paper]{geometry}
\usepackage{amsmath,amsthm,amsfonts,amssymb}
\usepackage{hyperref}
\usepackage[all]{xy}
\usepackage{graphicx}
\usepackage{amscd}
\usepackage{pb-diagram}
\usepackage{pstricks}
\usepackage[all]{xy}
\usepackage{yhmath}

\hypersetup{
    colorlinks=true,
    linkcolor=blue,
    citecolor=green,
    urlcolor=cyan
}

\numberwithin{equation}{section}

\newtheorem{theorem}{Theorem}[section]
\newtheorem{definition}[theorem]{Definition}

\newtheorem{lemma}[theorem]{Lemma}
\newtheorem{remark}[theorem]{Remark}
\newtheorem{prop}[theorem]{Proposition}

\newtheorem*{nonothm}{Theorem}

\renewenvironment{proof}[1][\proofname. ]{ { \noindent \it #1}}{\qed \\}

\newcommand{\AI}{A_\infty}
\newcommand{\CG}{\mathcal G}

\newcommand{\LL}{\mathbb{L}}
\newcommand{\RR}{\mathbb{R}}
\newcommand{\CC}{\mathbb{C}}
\newcommand{\ZZ}{\mathbb{Z}}

\newcommand{\PP}{\mathbb{P}}
\newcommand{\bk}{\boldsymbol{k}}

\newcommand{\CE}{\mathcal{E}}
\newcommand{\CF}{\mathcal{F}}

\newcommand{\CM}{\mathcal{M}}

\newcommand{\Hom}{{\rm Hom}}

\newcommand{\locmir}{\mathcal{L}\mathcal{M}}
\newcommand{\calC}{\mathcal{C}}
\newcommand{\CO}{\mathcal O}

\newcommand{\Ext}{{\rm Ext}}

\newcommand{\id}{{\rm id}}
\newcommand{\Proj}{{\rm Proj}}
\newcommand{\CS}{\mathcal{S}}
\newcommand{\ind}{{\rm ind}}
\newcommand{\calA}{{\mathcal{A}}}

\begin{document}
\title{Comparison of mirror functors of elliptic curves via LG/CY correspondence}
\author{Sangwook Lee}

\begin{abstract}
Polishchuk-Zaslow explained the homological mirror symmetry between Fukaya category of symplectic torus and the derived category of coherent sheaves of elliptic curves via Lagrangian torus fibration. Recently, Cho-Hong-Lau found another proof of homological mirror symmetry using localized mirror functor, whose target category is given by graded matrix factorizations. We find an explicit relation between these two approaches.
\end{abstract}
%

\address{Korea Institue for Advanced Study, Hoegiro 85, Dongdaemun-gu, Seoul 02455,
Korea}
\email{swlee@kias.re.kr}
\maketitle
\tableofcontents
\section{Introduction}
Homological Mirror Symmetry(HMS) conjecture by Kontsevich has been a powerful motivation in recent developments of geometry and physics. Inspired by string theory, Kontsevich conjectured the equivalence of the derived Fukaya category of a Calabi-Yau manifold $X$ and the derived category of coherent sheaves of the other Calabi-Yau manifold $\check{X}$, which is called the mirror of $X$.

 The elliptic curve case was studied by Polishchuk-Zaslow\cite{PZ}. Then Seidel\cite{Sei3} proved the conjecture for the quartic surface. Also Abouzaid-Smith\cite{AS} proved homological mirror symmetry for higher-dimensional(in particular 4-dimensional) tori. Many more important works has followed, but we will not mention them further.

On the other hand, inspired by the work of Seidel on genus two curve \cite{Sei2}, 
Cho-Hong-Lau\cite{CHL1} developed, so called {\em localized mirror functors} formalism, and applied it to the study of HMS for orbifold spheres. Their idea is to think of an immersed Lagrangian submanifold $\mathbb{L}$ in a symplectic orbifold, and consider the Maurer-Cartan solutions of its $\AI$-algebra whose weak bounding cochains are given by immersed sectors. The superpotential which given by the counting of decorated polygons is a (quasi)homogeneous polynomial $W$. 
 Then an explicit homological mirror functor is constructed by considering the (curved) Yoneda functor $\locmir^\LL(\cdot):=CF^*(\LL,\cdot)$, which gives an $\AI$-functor $Fu_0(X) \to MF(W)$. Here $Fu_0$ means the subcategory whose objects are unobstructed Lagrangians. In this correspondence the Floer complex $CF^*(\LL,L)$ for an unobstructed Lagrangian $L$, directly gives a matrix factorization of $W$. Taking twisted complexes and cohomologies on both sides, we get an exact functor.

From now on we concentrate on the HMS of elliptic curves.
Categorical mirror symmetry of Polishchuk-Zaslow(\cite{PZ}) compares the derived category of coherent sheaves of an elliptic curve $X$ and derived Fukaya category of a symplectic torus $T^2$. This foundational work gave a first non-trivial example of homological mirror symmetry. Roughly, they matched line bundles of degree $d$ on $X$ to the lines of slope $d$ in $X$ both of which may come with additional data (tensoring with higher dimensional bundles and flat connections on bundles respectively). Intersections between lines translates to theta functions and the Floer product corresponds to theta identities.

In Cho-Hong-Lau \cite{CHL1}, one first considers the the symplectic torus $T^2$ (with $\ZZ/3$-symmetry) and its $\ZZ/3$-quotient orbifold $\PP^1_{3,3,3}$. The immersed Lagrangian in this orbifold (called Seidel Lagrangian) defines the Landau-Ginzburg mirror $(\Lambda^3,W)$ with an $\AI$-functor from Fukaya category of $\PP^1_{3,3,3}$ to the dg-category of matrix
factorizations. To recover the mirror symmetry of the original symplectic torus $T^2$, one need to take the $\ZZ$-graded version of this functor from graded Fukaya category of $T^2$  to the graded matrix factorization category of $MF_{\ZZ}(W)$ (\cite{CHL2})


Hence, we have two different kinds of homological mirror symmetries of elliptic curves. They have $B$-model categories as a derived category and a category of matrix factorizations, respectively. 
Indeed, by Orlov's theorem\cite{Or}, these two categories are equivalent. Namely, if $W$ is a (quasi-)homogeneous polynomial which defines a smooth projective CY hypersurface $X$, then $D^b Coh(X) \simeq HMF_\ZZ(W)$. This equivalence is called a {\em Landau-Ginzburg/Calabi-Yau}(LG/CY for short) {\em correspondence}.

So far we have different exact functors between triangulated categories, and we can ask how they are related to each other. We find an explicit relation as follows.





\begin{nonothm}
We have a commutative diagram of exact functors
\[\xymatrix{
D^\pi Fu(T^2) \ar[r]^{\CS_i} \ar[d]^\Phi & D^\pi Fu(T^2)\ar[d]^{\locmir^\LL} \\
D^b Coh(X) \ar[r]^{\CG_i} & HMF_\ZZ(W).}
\]
where $\Phi$ is the mirror functor of \cite{PZ,AS} and 
\begin{equation}\label{si}\CS_i=[-j]\circ \tau^d\circ t_{(0,1/2)}\circ \left(\begin{array}{cc}1 & 0 \\ -3i+2 & 1\end{array}\right)\end{equation}
where $j=\lfloor -\frac{i}{3} \rfloor$, $d=-i-3j$, $\tau$ is the rotation by $-2\pi/3$ and $t_{(0,1/2)}$ is the parallel transport by $(0,1/2)$.
\end{nonothm} 
\begin{remark}
The definition of $\locmir^\LL$ involves a choice of a character $\gamma:\ZZ/3 \to U(1)$. Varying the choice of $\gamma$, the functors $\CS_i$ may also vary. Here we have fixed one choice.
\end{remark}
Namely, two homological mirror symmetry are equivalent after certain geometric transformation $S_i$ (rotation and translation) and shifts.

We remark that in \cite{CHL2}, non-commutative homological mirror symmetry of elliptic curve has been discussed (whose mirror is given by non-commutative Landau-Ginzburg model, which is a choice of central element $W$ in Sklyanin algebra). 
The relation between commutative and non-commutative mirror functors is not known, and we hope to apply the method of this paper to compare commutative and non-commutative mirror functors in the future. 

Let us comment on the proof of the theorem.
Recall that Orlov's argument is based on the fact that $D^b Coh(X)$ and $HMF_\ZZ(W)$ are Verdier quotients of $D^b({\rm gr-}A)$. Instead of considering quotients of $D^b({\rm gr-}A)$ itself, consider quotients of its subcategory as 
\[\pi_i: D^b({\rm gr-}A_{\geq i}) \hookrightarrow D^b({\rm gr-}A) \to D^b({\rm gr-}A)/D^b({\rm tors-}A)\simeq D^b Coh(X),\]
\[q_i: D^b({\rm gr-}A_{\geq i}) \hookrightarrow D^b({\rm gr-}A) \to D^b({\rm gr-}A)/{\rm Perf-}A\simeq HMF_\ZZ(W).\]
Then Orlov constructed adjoint functors of above ones:
\[\RR\omega_i: D^b Coh(X) \to D^b({\rm gr-}A_{\geq i}),\]
\[\nu_i: HMF_\ZZ(W) \to D^b({\rm gr-}A_{\geq i}).\]
Then he proves that $\pi_i \circ \nu_i : HMF_\ZZ(W) \stackrel{\sim}{\longrightarrow} D^b Coh(X)$ and $\CG_i:=q_i \circ \RR\omega_i:D^bCoh(X) \stackrel{\sim}{\longrightarrow} HMF_\ZZ(W)$ if $W$ defines a CY variety.


The most nontrivial part for the proof is to compute images of $\RR\omega_i$ which is a right derived functor. It is not enough to know only cohomologies of the images, but we need to know them as genuine $\RR\Hom$-complexes precisely, because we need to compare morphisms between them, not just objects themselves. The scheme for the computation is to use the Gorenstein property of the ring $R/W$, because it can be used to make many terms in the cohomology long exact sequence vanish, so that the object which we suspect to be an image of $\CG_i$ is in the subcategory which is a component of the semiorthogonal decomposition, hence it is indeed an image of $\CG_i$. We also remark that when we compare morphisms of matrix factorizations we do not have to compute all entries, but it is sufficient to compare constant entries which in fact determine the morphism completely. This observation considerably reduces the counting of holomorphic strips.


The organization of the paper is as follows. In Section \ref{sec:Fuk} we recall basic ingredients of Fukaya categories. In Section \ref{sec:mf} we introduce the notion of graded matrix factorizations and relate them with a quotient of a derived category. Then we relate derived categories and matrix factorizations by recalling Orlov's LG/CY correspondence. In following two sections we introduce two different kinds of mirror symmetry of elliptic curves. Finally in Section \ref{sec:mainthm} we prove our main theorem.

\bigskip

{\bf Acknowledgements.} The author thanks Cheol-hyun Cho for the encouragement and a lot of helpful suggestions. He also thanks Hansol Hong and Siu-Cheong Lau for generously sharing their ideas and results. He is grateful to Yong-Geun Oh for his interests in this problem and a number of useful comments. He thanks Dohyeong Kim and Dong Uk Lee for letting him to care about crucial issues about elliptic curves. He is grateful to the Center for Geometry and Physics(IBS) for its hospitality and support when he worked on this paper as a postdoctoral research fellow of the center. This work was supported by IBS-R003-D1.

\section{Fukaya categories}\label{sec:Fuk}
We recall the definitions and relevant theorems of $\AI$-categories and Lagrangian Floer theory mainly to set the notations (we refer readers to \cite{FOOO}, \cite{Aur} for example).
\subsection{Filtered $\AI$-categories}
\begin{definition}
The {\em Novikov field} is $\displaystyle\Lambda:=\Big\{\sum_{i \geq 0} a_i T^{\lambda_i} \mid a_i \in \CC, \lambda_i \in \RR, \lambda_i \to \infty \;\mathrm{as} \; i \to \infty \Big\}.$

\end{definition}

A filtration $F^\bullet\Lambda$ of $\Lambda$ is given by \[F^\lambda \Lambda:=\Big\{\sum_{i\geq 0}a_i T^{\lambda_i} \mid \lambda_i \geq \lambda {\rm \;for\; all\; }i \Big\} \subset \Lambda, F^+\Lambda:=\Big\{\sum_{i\geq 0}a_i T^{\lambda_i} \mid \lambda_i > 0 {\rm \;for\; all\; }i \Big\}.\]
The {\em Novikov ring} $\Lambda_0$ is defined as $\Lambda_0:= F^0 \Lambda$.

\begin{definition}
A {\em filtered $\AI$-category} $\calC$ over $\Lambda$ consists of a class of objects $Ob(\calC)$ and the set of morphisms $hom_{\calC}(A,B)$ for a pair of objects $A,B$ of $\calC$ with the following conditions:
\begin{enumerate}
 \item $hom(A,B)$ is a filtered $\ZZ$-graded $\Lambda$-vector space for any $A,B\in Ob(\calC)$,
 \item for $k \geq 0$ there are multilinear maps of degree 1
  \[m_k: hom(A_0,A_1)[1] \otimes hom(A_1,A_2)[1] \otimes \cdots \otimes hom(A_{k-1},A_k)[1] \to hom(A_0,A_k)[1]\]
  such that they preserve the filtration and satisfy the {\em $\AI$-relation}
  \[\sum_{k_1+k_2=n+1}\sum_{i=1}^{k_1}(-1)^\epsilon m_{k_1}(x_1,...,x_{i-1},m_{k_2}(x_i,...,x_{i+k_2-1}),x_{i+k_2},...,x_n)=0\] where $\epsilon=\sum_{j=1}^{i-1}(|x_j|+1).$ 
   
\end{enumerate}
 Here, $m_0$ means that for each object $A$ we have $m_0^A \in hom^1(A,A)[1]=hom^2(A,A).$
If $m_0 \neq 0$, $\calC$ is called a {\em curved} $\AI$-category. Otherwise, $\calC$ is called {\em strict.}
If there is only one object, then $\calC$ is called an {\em $\AI$-algebra.} If only $m_1$ and $m_2$ are nonzero, then $\calC$ is called a {\em dg category.}

\end{definition}
In this paper, every $\AI$-category is filtered over $\Lambda$. $\AI$-categories are generalizations of dg categories where composition of morphisms may be associative only up to homotopy.

To understand the meaning of the $\AI$-relation with possibly nonzero $m_0$, we write down the relation for the simplest case. For $x \in hom(A,B)$, 
\begin{equation}\label{eq:m1}m_1^2(x)+m_2(m_0^A,x)+(-1)^{|x|+1}m_2(x,m_0^B)=0.\end{equation}
Hence if $m_0 \neq 0$, $m_1$ may not be a differential (i.e. $m_1^2=0$).
\begin{definition}
For an object $A$ in an $\AI$-category, $e_A \in hom(A,A)$ is called a {\em unit} if it satisfies
 \begin{enumerate}
  \item $m_2(e_A,x)=x,\; m_2(y,e_A)=(-1)^{|y|}y$ for any $x \in hom(A,B)$, $y \in hom(B,A)$,
  \item $m_{k+1}(x_1,...,e_A,...,x_k)=0$ for any $k \neq 1$.
 \end{enumerate}
\end{definition}
we recall the deformation theory of $\AI$-category. 
\begin{definition}
An element $b \in F^+hom^1(A,A)$ is called a {\em weak bounding cochain} of $A$ if it is a solution of the weak Maurer-Cartan equation
\begin{equation}\label{eq:weakMC}
m(e^b):=m_0^A+m_1(b)+m_2(b,b)+ \cdots = PO(A,b)\cdot e_A\end{equation}
for some $PO(A,b) \in \Lambda.$ If such a solution exists, then $A$ is called {\em weakly unobstructed}. 
 If there exists a solution $b$ such that $PO(A,b)=0$, then $b$ is called a {\em bounding cochain} and $A$ is called {\em unobstructed}. $PO(A,b)$ is called the {\em Landau-Ginzburg superpotential} of $b$. \end{definition}
We denote $\mathcal{M}_{weak}(A)$ be the set of weak bounding cochains of $A$. Then $PO(A,\cdot)$ is a function on $\mathcal{M}_{weak}(A)$. We also define 
\[\CM_{weak}^\lambda(A):=\{b\in \CM_{weak}(A) \mid PO(A,b)=\lambda\}.\]

Following Proposition 1.20 of \cite{Fuk}, given an $\AI$-category $\calC$, under the assumption $\CM^\lambda_{weak}(A)$ is nonempty for some objects, we define a new $\AI$-category $\calC_\lambda$ as
\[Ob(\calC_\lambda)=\bigcup_{A \in Ob(\calC)}\{A\} \times \CM_{weak}^\lambda(A),\]
\[hom_{\calC_\lambda}((A_1,b_1),(A_2,b_2))=hom_{\calC}(A_1,A_2)\]
with the following $\AI$-structure maps
\[m_k^{b_0,...,b_k}:hom_{\calC_\lambda}((A_0,b_0),(A_1,b_1)) \otimes \cdots \otimes hom_{\calC_\lambda}((A_{k-1},b_{k-1}),((A_k,b_k)) \to hom_{\calC_\lambda}((A_0,b_0),(A_k,b_k)),\]
\[m_k^{b_0,...,b_k}(x_1,...,x_k):=\sum_{l_0,...,l_k}m_{k+l_0+\cdots+l_k}(b_0^{l_0},x_1,b_1^{l_1},...,b_{k-1}^{l_{k-1}},x_k,b_k^{l_k})\]
where $x_i \in hom_{\calC_\lambda}((A_i,b_i),(A_{i+1},b_{i+1})).$ $\AI$-relation is induced by the weak Maurer-Cartan equation (\ref{eq:weakMC}).

\begin{theorem}
Let $(A_0,b_0),(A_1,b_1) \in Ob(\calC_\lambda)$. Then $(m_1^{b_0,b_1})^2=0.$
\end{theorem}

\begin{proof}
Let $x \in hom_{\calC_\lambda}((A_0,b_0),(A_1,b_1)).$ Then the $\AI$-equation is \[(m_1^{b_0,b_1})^2+m_2(m(e^{b_0}),x)+(-1)^{|x|+1}m_2(x,m(e^{b_1}))=0.\] By $m(e^{b_0})=\lambda\cdot e_{A_0},$ $m(e^{b_1})=\lambda \cdot e_{A_1}$ and by definition of units, \[m_2(m(e^{b_0}),x)+(-1)^{|x|+1}m_2(x,m(e^{b_1}))=0,\] so $(m_1^{b_0,b_1})^2=0.$
\end{proof} 
So, under the existence of the weak Maurer-Cartan solutions, we get strict $\AI$-categories by restricting to objects sharing certain value of the Landau-Ginzburg(LG for short) superpotential.

\begin{definition}
Let $\calC$ and $\calC'$ be $\AI$-categories. An {\em $\AI$-functor} between $\calC$ and $\calC'$ is a collection $\CF=\{\CF_i\}_{i \geq 0}$ consisting of

\begin{itemize}
 \item $\CF_0: Ob(\calC) \to Ob(\calC'),$
 \item $\CF_k: hom_{\calC}(A_0,A_1)[1] \otimes \cdots \otimes hom_{\calC}(A_{k-1},A_k)[1] \to hom_{\calC'}(\CF_0(A_0),\CF_0(A_k))[1]$ of degree 0
\end{itemize}
which are subject to the following condition:
\begin{eqnarray*}
& \sum\limits_{i,j} (-1)^{|x_1|'+\cdots+|x_{i-1}|'} \CF_{i-j+k}(x_1,...,x_{i-1},m^\calC_{j-i+1}(x_i,...,x_j),x_{j+1},...,x_k)\\
=& \sum\limits_l m^{\calC'}_{l+1}(\CF_{i_1-1}(x_1,...,x_{i_1}),\CF_{i_2-i_1}(x_{i_1+1},...,x_{i_2}),...,\CF_{k-i_l}(x_{i_l+1},...,x_{k})).
\end{eqnarray*}

\end{definition}

\subsection{Triangulated $\AI$-categories}
By \cite{Sei1}, we know that any $\AI$-category $\calC$ admits a cohomologically fully faithful functor into another $\AI$-category which is called a {\em triangulated envelope} of $\calC$, in which we have exact triangles and shift functors. 
We take the most common construction of triangulated envelope given by so-called {\em twisted complexes.} Since we do not use non-trivial twisted complex in this paper, we omit its precise definition(and refer readers to \cite{Sei1}) and just give a short summary: given an $\AI$-category $\calC$ we add formal shifts and formal direct sums, and equip an object $\displaystyle E=\bigoplus_{i=1}^N E_i[k_i]$ with a strictly lower triangular map $\delta:E \to E$ such that $\displaystyle\sum_{k \geq 1} m_k(\delta,...,\delta)=0.$ Then the pair $(E,\delta)$ is called a {\em twisted complex}. Morphisms among them and $\AI$-structure maps are defined in the most canonical way, and denote the resulting $\AI$-category by $Tw(\calC).$

\subsection{Fukaya category on surfaces}\label{subsec:fukaya}
We will use a version of Fukaya category of surface $M$ described in \cite{Sei2}
with a different set of conventions(as used in \cite{CHL1}). We recall relevant ingredients for readers convenience. Roughly, Fukaya category of a symplectic manifold $M$(denoted by $Fu(M)$) is an $\AI$-category whose objects are Lagrangian submanifolds with additional data and with morphisms given by Floer complexes. For simplicity, assume that $L$ and $L'$ are oriented spin Lagrangian submanifolds which intersect transversely. Then the {\em Floer complex} $CF(L,L')$ is a vector space over $\Lambda$ whose basis elements are intersections of $L$ and $L'$. 
Each intersection has an associated index (or parity in $\ZZ/2$-graded case) defined as follows. For $p \in L \cap L'$, 
choose a smooth path of oriented Lagrangian subspaces $\lambda_{p_0,p_1}(t)$ for $0 \leq t \leq 1$ at $p$, $\lambda_{p_0,p_1}(0)=T_p L$ and $\lambda_{p_0,p_1}(1)=T_p L'$. Then concatenate the positive definite path $\gamma$ from $T_p L'$ to $T_p L$, which does not depend on the orientation of Lagrangians. The homotopy class of the loop $\gamma \circ \lambda_{p_0,p_1}$ in Lagrangian Grassmannian from $T_p L$ to itself gives a winding number, which is called the {\em degree of $\lambda_{p_0,p_1}$}.
Here, a positive definite path from $T_p L_0$ to $T_p L_1$ is defined by identifying $L_0 \cong \RR^n$ and $L_1 \cong i\RR^n$ at $p=(0,0,...,0)$ and taking the path $exp(\pi i t) \cdot \RR^n$ for $0 \leq t \leq 1/2$(see RHS of Figure 2 for 1-dimensional case of the positive path).

The parity of the degree does not depend on the choice of the path $\lambda_{p_0,p_1}(t)$, and also if we consider a Calabi-Yau manifold $M$ and its graded Lagrangians, then there is a canonical Lagrangian path between $T_p L$ and $T_p L'$, in the sense that the path should preserve phases. Recall that an oriented Lagrangian submanifold in a CY manifold $(M,\omega,\Omega)$ is {\em graded} if there is a function $\theta_L:L \to \RR$ such that \[\displaystyle\frac{\Omega(X_1(p)\wedge \cdots \wedge X_n(p))}{|\Omega(X_1(p) \wedge \cdots \wedge X_n(p))|}=e^{i\theta_L(p)}\] for any positively oriented wedge of vector fields $X_1 \wedge \cdots \wedge X_n$ of $L$. $\theta_L$ is called the {\em phase function} of $L$. If $\theta_L$ is constant, then $L$ is called {\em special Lagrangian}.
Hence in the graded case the degree of each intersection point is well-defined in $\ZZ$. 

\begin{figure}
\includegraphics[height=2in]{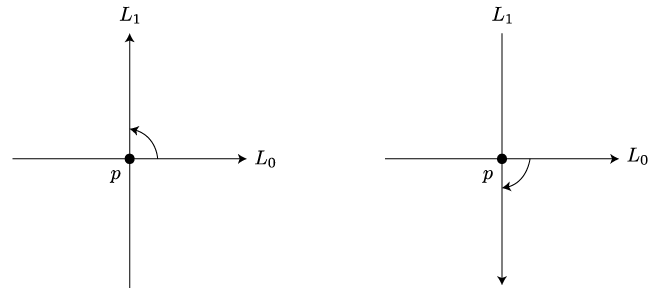}
\caption{The left picture is a path from $T_p L_0$ with phase $0$ to $T_p L_1$ with phase $\frac{\pi}{2}.$ In this case $\deg (p)=1.$ The right one is a path from $T_p L_0$ with phase $0$ to $T_p L_1$ with phase $-\frac{\pi}{2}$, and $\deg(p)=0.$}
\end{figure}

Floer differential $m_1: CF(L,L') \to CF(L,L')$ is defined as 
\[m_1(p):=\sum_{\stackrel{q\in L\cap L'}{\ind([u])=1}}\#(\CM(p,q;[u]))T^{\omega(u)}q\]
where $u$ is a $J$-holomorphic strip $u: \RR \times [0,1]\to M$ with
\[u(s,1) \in L, \; u(s,0) \in L',  \lim_{s \to -\infty}u(s,t)=p,\; \lim_{s \to \infty}u(s,t)=q.\]
And $\#$ is a signed count and $\omega(u)$ is the symplectic area of $u$. 
The {\em index} of the strip $u$ is defined by the Maslov index. 
Higher $\AI$-operations on morphisms $m_k: CF(L_0,L_1) \otimes \cdots \otimes CF(L_{k-1},L_k) \to CF(L_0,L_k)$ is defined by counting $J$-holomorphic polygons.

Let $p_i \in CF(L_{i-1},L_i)$ and $q \in CF(L_0,L_k).$ We define a moduli space $\CM(p_1,...,p_k;q)$ of $J$-holomorphic polygons whose domains are $D^2$ minus $k+1$ boundary points cyclically ordered by $z_1,...,z_k,z_0$, arcs between $p_i$ and $p_{i+1}$ are mapped inside $L_i$(and inside $L_k$ between $p_k$ and $q$), and the images near those punctures are asymptotically $p_1,...,p_k,q$, respectively. Let $\CM(p_1,...,p_k;q;\beta)$ be a subset of $\CM(p_1,...,p_k;q)$ which consists of holomorphic polygons of homotopy class $\beta$. Then the dimension of the moduli space is given by
\[\dim \CM(p_1,...,p_k;q;\beta)=k-2+\ind(\beta).\]
The index of $u \in \CM(p_1,...p_k;q)$ is also given by the Maslov index. Fix a trivialization of $u^*TM$ so that we get paths of Lagrangian subspaces $l_0,l_1,...,l_k$ on $L_0,L_1,...,L_k$ respectively. Then we start from $T_{p_1}L_0$, at corners $p_i$ concatenate negative definite paths, move along $l_i$ until we arrive at $T_q L_k.$ At $q$ concatenate the positive definite path and move along $l_0$ to arrive at $T_{p_1} L_0$ again. The index of $u$ is defined by the Maslov index of the loop described above, and it depends only on the homotopy class of $u$. Now we define \[m_k: CF(L_0,L_1) \otimes \cdots \otimes CF(L_{k-1},L_k) \to CF(L_0,L_k)\] by
\[m_k(p_1,...,p_k):=\sum_{\stackrel{p_i\in L_{i-1}\cap L_i,q\in L_0\cap L_k}{\ind[u]=2-k}}\# (\CM(p_1,...,p_k;q;[u]))T^{\omega(u)}q.\]

\begin{figure}
\includegraphics[height=2in]{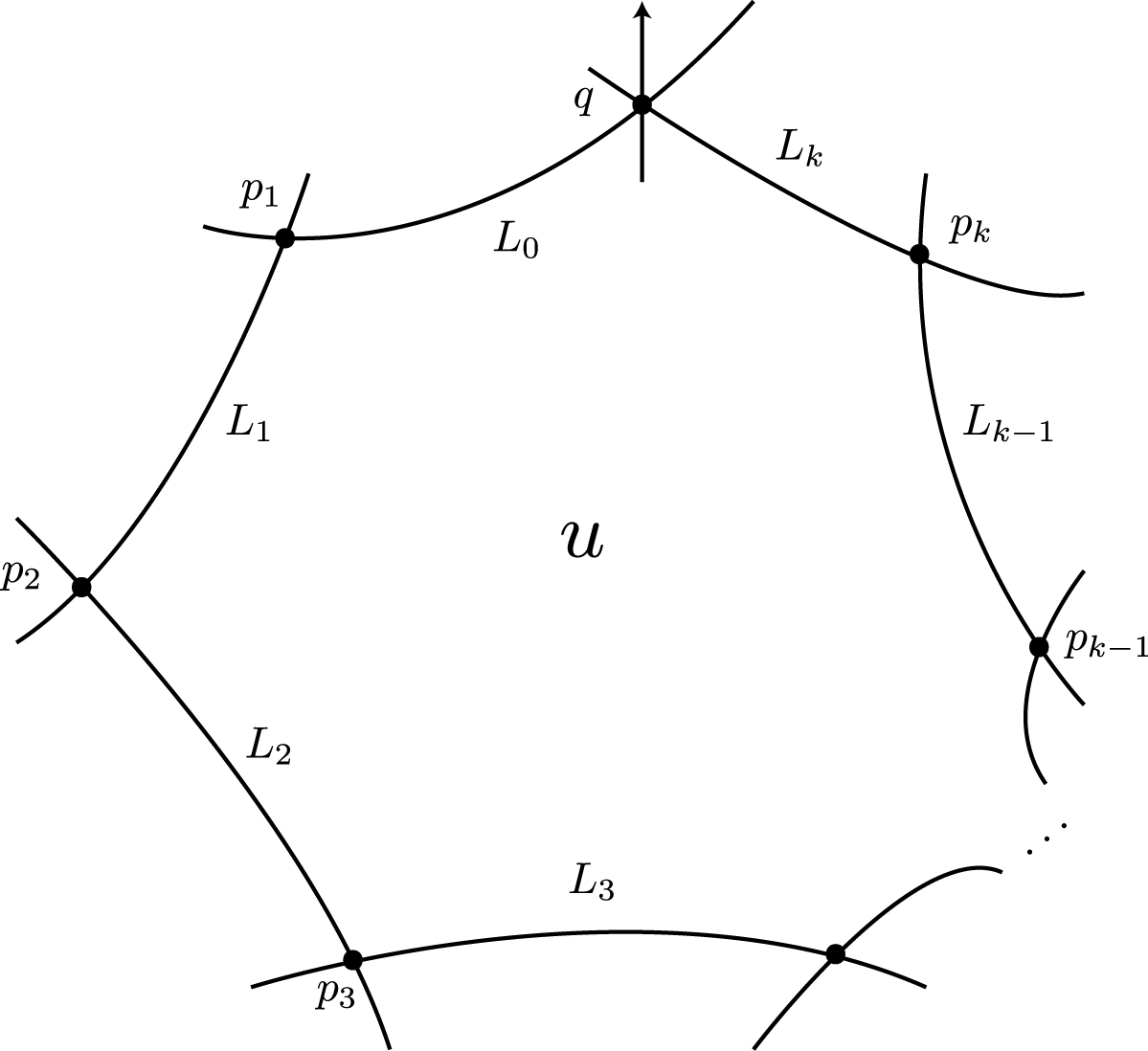}
\caption{Holomorphic polygon $u$ contributing to $m_k(p_1,...,p_k)$.}
\end{figure}

Recall that in graded case we can define degrees of Lagrangian intersections in $\ZZ$. Then if a holomorphic polygon $u$ has corners $p_1\in L_0 \cap L_1,...,p_k\in L_{k-1}\cap L_k,q\in L_0 \cap L_k$, then \[\ind(u)=\deg(q)-\deg(p_1)-\cdots-\deg(p_k)\] where $q \in CF(L_0,L_k)$.

Now let us consider the case of surfaces. The precise construction of Fukaya category is more involved since one has to deal with non-transverse Hom spaces $CF(L,L)$. In \cite{Sei2}, a Morse function on $S^1$ has been chosen so that the Hom space is generated by critical points. We refer readers to \cite{Sei2} for further discussions.
Let us recall the definition of orientation for the counting of polygons from \cite{Sei2}. Let $u \in \CM(p_1,...,p_k;q)$ whose boundary lies on Lagrangian submanifolds as above. The sign of $u$ is determined by the following steps.
\begin{itemize}
 \item If a Lagrangian is equipped with a nontrivial spin structure, put a point $\circ$ on it, on which the nontrivial spin bundle is twisted.
 \item Disagreement of the orientation of $\partial u$ on $L_0$ is irrelevant.
 \item If the orientation of $\partial u$ on $\wideparen{p_i p_{i+1}}$ does not agree with $L_i$, the sign is affected by $(-1)^{|p_i|}.$
 \item If the orientation of $\partial u$ on $\wideparen{p_k q}$ does not agree with $L_k$, the sign is affected by $(-1)^{|p_k|+|q|}.$ 
 \item Mutiply $(-1)^l$ when $\partial u$ passes through nontrivial spin points $\circ$ $l$ times.
\end{itemize}
The structure maps $\{m_k\}_{k \geq 0}$ define an $\AI$-structure, and the resulting $\AI$-category is called the {\em Fukaya category} of $M$ and written as $Fu(M)$. In general Fukaya category may be obstructed, i.e. $m_0$ is not zero, so $CF(L,L')$ might not be a chain complex. But if we form an $\AI$-subcategory $Fuk_\lambda(M)$ of weakly unobstructed objects equipped with weak bounding cochains whose LG superpotentials have same value $\lambda$, then $m_1$ on $Fuk_\lambda(M)$ is a differential, and if $(L,b),(L',b')\in Fuk_\lambda(M)$, the cohomology of $(CF((L,b),(L',b')),m_1)$ is called the {\em Floer cohomology} of the pair $((L,b),(L',b')),$ denoted by $HF((L,b),(L',b')).$ In particular, $Fu_0(M)$ is an $\AI$-subcategory of $Fu(M)$ of unobstructed objects. We remark another important fact that $CF(L,L')$ is homotopy equivalent to $CF(L,\phi(L'))$ if $\phi$ is a Hamiltonian diffeomorphism. In particular, if $L=L'$, then $CF^*(L,L) \cong CF^*(L,\phi(L))$ for any Hamiltonian diffeomorphism $\phi$, and $HF^*(L,L) \cong H^*(L,\Lambda)$. By definition, weak bounding cochains of $L$ are in $CF^1(L,L).$

\subsection{Derived Fukaya categories}
Since $Fu(M)$ is an $\AI$-category, we also have its triangulated envelope $Tw(Fu(M))$ by adding twisted complexes of Lagrangians, and taking its cohomology category, we get the {\em derived} Fukaya category $DFu(M).$ Taking split-closure, we get the split-closed derived Fukaya category $D^\pi Fu(M).$ 

We will be mainly concerned with direct sums of Lagrangian submanifolds with new kinds of bounding cochains which occur by intersections between direct summands. First we clarify the meaning of direct sums and shifts in Fukaya categories. A direct sum of Lagrangian submanifolds is just the union of them. It can be also considered to be an immersed Lagrangian. Given an object $A$ in a triangulated $\AI$-category, $A[1]$ is featured by the property $hom^i(A[1],B)\cong hom^{i-1}(A,B)$, $hom^i(B,A[1])\cong hom^{i+1}(B,A)$. Hence, by definition of the degree of morphisms(or intersections) between Lagrangian submanifolds, in non-graded case $[1]$ is just reversing of the orientation. In 1-dimensional graded case in which we are interested, it corresponds to the change of phase by $-\pi$.

If $L=L_1 \oplus \cdots \oplus L_n$, then $\displaystyle CF^*(L,L) \cong \bigoplus_{1\leq i,j \leq n}CF^*(L_i,L_j).$ Then there is another kind of degree 1 cochains which are given by $CF^1(L_i,L_j)$ with $i\neq j$, in addition to those of degree 1 cochains of a single embedded Lagrangian submanifold. In particular if $L$ is 1-dimensional and all $L_i$ are transverse to each other without triple(or more multiple) intersections, then some intersections among them are degree one cochains, and furthermore they can contribute to be a part of weak bounding cochains. We will encounter such an example later, namely (lifts of) Seidel Lagrangian on $T^2$.

\section{Graded matrix factorizations}\label{sec:mf}
Let $R=\Lambda[x_0,...,x_n]$ be a graded ring with $\deg(x_i)=d_i$.
\begin{definition}
Let $W \in R$ be a (quasi)homogeneous polynomial of degree $d$. $MF_\ZZ(W)$ is a dg category whose object $(P,d_P)$ is represented as a pair of graded morphisms $p_0: P_0 \to P_1$ and $p_1: P_1 \to P_0(d)$, where $P_0$ and $P_1$ are graded free $R$-modules and
\[p_0(d) \circ p_1=W \cdot \id : P_1 \to P_1(d),\]
\[p_1 \circ p_0 = W \cdot \id: P_0 \to P_0(d).\] 

Equivalently, an object described above is also expressed as a quasi-periodic infinite sequence
\[K^\cdot: \xymatrix{ \cdots \ar[r] & K^i \ar[r]^{k^i} & K^{i+1} \ar[r]^{k^{i+1}}& K^{i+2} \ar[r] & \cdots}\]
where $K^{2i+1}=P_1(i\cdot d)$, $K^{2i}=P_0(i\cdot d),k^{2i}=p_0(i\cdot d), k^{2i+1}=p_1(i\cdot d).$ Then 
\[hom^j_{MF_\ZZ(W)}(K^\cdot,L^\cdot):=\Big\{f^\cdot: K^\cdot \to L^{\cdot+j},{\rm \; graded}\mid f^{i+2}=f^{i}(d) \Big\}\]
and $d: hom^j_{MF_\ZZ(W)}(K^\cdot,L^\cdot) \to hom^{j+1}_{MF_\ZZ(W)}(K^\cdot,L^\cdot)$ is defined by
\[(d(f^\cdot))^i:=l^{i+j} \circ f^i + (-1)^j f^{i+1} \circ k^i\]
where $l^n:L^n \to L^{n+1}.$ Compositions are defined as usual.


\end{definition}

Observe that given a matrix factorization $K^\cdot$ it is natural to define the shift $K^\cdot[1]$ such that $K[1]^i=K^{i+1}$ and $k[1]^i=-k^{i+1}.$ It is clear that $K^\cdot[2]=K^\cdot(d).$

\begin{definition}
Given a dg category $\calC$, its {\em cohomology category} $H^0(\calC)$ is defined by the same objects as those of $\calC$, and morphism spaces as 0th cohomologies of $d:hom^j \to hom^{j+1}.$

\end{definition}

\begin{prop}
The cohomology category $H^0(MF_\ZZ(W))$ is a triangulated category with exact triangles 
\xymatrix{K^\cdot \ar[r]^f & L^\cdot \ar[r] &C^\cdot(f) \ar[r] & K^\cdot[1]} where the mapping cone of $f$ is defined as 
\[C^\cdot(f): \xymatrix{ \cdots \ar[r] & L^i \oplus K^{i+1} \ar[r]^{c^i} & L^{i+1}\oplus K^{i+2} \ar[r]^{c^{i+1}} & L^{i+2} \oplus K^{i+3} \ar[r] & \cdots}\] such that \[c^i=\left(\begin{array}{cc}l^i & f^{i+1} \\0 & -k^{i+1}\end{array}\right).\]

\end{prop}
The proof is standard, as in the case of homotopy categories over abelian categories.

\begin{definition}
We write $HMF_\ZZ(W):=H^0(MF_\ZZ(W))$ and call it the {\em category of graded matrix factorizations} of $W$.
\end{definition}

\begin{remark}
The category $HMF_\ZZ(W-\lambda)$ is nontrivial(i.e. it contains nonzero objects) only when $\lambda$ is a critical value. Since we only deal with homogeneous polynomials, 0 is a critical value(a degree 1 polynomial does not admit any nontrivial matrix factorization), so the category we are interested in this paper is nontrivial.
\end{remark}

Let $A=R/W$. Since $W$ is homogeneous, $A$ is a graded ring. Then the category ${\rm gr-}A$ of finitely generated graded $A$-modules is an abelian category. We define ${\rm Perf-}A$ as the full subcategory of chain complexes of $A$-modules, which are quasi-isomorphic to bounded complexes of projectives. Then ${\rm Perf-}A$ is a thick subcategory of $D^b({\rm gr-}A).$ We recall a useful lemma.
\begin{lemma}
$HMF_\ZZ(W) \simeq D^{gr}_{sg}(A)$, where $D^{gr}_{sg}(A) := D^b({\rm gr-}A)/{\rm Perf-}A.$
\end{lemma}
We omit the proof but explain its origin. Since $A$ is singular, the minimal $A$-free resolution of an object in $D^b({\rm gr-}A)$ need not terminate, but it eventually become quasi-(2-)periodic by Eisenbud's theorem \cite{Eis}. If we replace free $A$-modules in the resolution by free $R$-modules of same ranks and consider differentials as morphisms of $R$-modules, then the asymptotic 2-periodic part indeed becomes a matrix factorization of $W$. If two objects in $D^b({\rm gr-}A)$ have free resolutions which are asymptotically the same, then they define equivalent object in $D^{gr}_{sg}(A)$ by definition, or equivalently, they give the same matrix factorization.
Finally, we remark that given a matrix factorization $K^\cdot$ we take ${\rm Cok}(k^{-1}:K^{-1} \to K^0)$ to obtain an object of $D^{gr}_{sg}(A).$

\section{Orlov's LG/CY correspondence}\label{sec:LGCY}
Let $X=\Proj(A)$ where $A=R/W$ as above, i.e. $R=\Lambda[x_0,...,x_n]$ and $W$ is a homogeneous polynomial. In this section we recall the correspondence between $HMF_\ZZ(W)$ and $D^bCoh(X)$ in \cite{Or}.
\begin{remark}$A$ is a {\em Gorenstein} algebra, i.e. it has finite injective dimension $n$ and if $D(\bk):=\RR \Hom_A(\bk,A)$ where $\bk \cong A/(x_0,...,x_n),$ (observe that $\bk \cong \Lambda$ as a vector space. Nevertheless, we distinguish the notation $\bk$ from $\Lambda$ because we want to emphasize that it is an $A$-module) then it is isomorphic to $\bk(a)[-n]$ for some integer $a$ which is called the {\em Gorenstein parameter}. This homological condition on $A$ enables us to construct various derived functors (in later sections) between {\em bounded} derived categories. Also, if $W$ defines a Calabi-Yau variety, then $a=0$. These properties will be crucially used in Section \ref{sec:mainthm}.
\end{remark}
The idea of LG/CY correspondence comes from the fact that two categories are both Verdier quotients of $D^b({\rm gr-}A)$. Let ${\rm tors-}A$ be the subcategory of ${\rm gr-}A$ of torsion modules, i.e. $A$-modules which are finite dimensional over $A_0 \cong \Lambda.$ By Serre's theorem, $D^bCoh(X) \simeq D^b({\rm qgr-}A):=D^b({\rm gr-}A)/D^b({\rm tors-}A)$, whereas $HMF_\ZZ(W) \simeq D^b({\rm gr-}A)/{\rm Perf-}A$ just as shown above. 

If the quotient functors $\pi: D^b({\rm gr-}A) \to D^b({\rm qgr-}A)$ and $q: D^b({\rm gr-}A) \to D^{gr}_{sg}(A)$ have adjoint functors, then we can lift objects and morphisms in a quotient category to those in $D^b({\rm gr-}A)$, and project them to the other quotient, and then we would obtain functors between two quotient categories. Unfortunately, $q$ does not admit any adjoint functor while $\pi$ admits a right adjoint, but if we consider restrictions $\pi_i: D^b({\rm gr-}A_{\geq i})  \hookrightarrow D^b({\rm gr-}A) \to D^b({\rm qgr-}A)$ and $q_i: D^b({\rm gr-}A_{\geq i}) \hookrightarrow D^b({\rm gr-}A) \to D^{gr}_{sg}(A)$, where ${\rm gr-}A_{\geq i}$ consists of modules $M$ such that $M_p=0$ for $p<i$, then $\pi_i$ still admits a right adjoint, and $q_i$ has a left adjoint. Now we describe the adjoint of $\pi_i$.

\begin{lemma}
Define $\RR \omega_i: D^b({\rm qgr-}A) \to D^b({\rm gr-}A_{\geq i})$ by \[\RR \omega_i(M):=\bigoplus_{k=i}^\infty \RR\Hom_{D^b({\rm qgr-}A)}(\pi A, M(k)).\]
Then $\RR\omega_i$ is fully faithful and $\RR\omega_i$ is the right adjoint to $\pi_i$. Moreover, all cohomologies $\RR^j \omega_i(M)$ are contained in ${\rm tors-}A$ for $j>0$.
\end{lemma}

We use the same notation for the functor $\RR\omega_i: D^b Coh(X) \to D^b({\rm gr-}A_{\geq i})$, 
\[\RR \omega_i(\mathcal{E}):=\bigoplus_{k=i}^\infty \RR\Hom_{D^b Coh(X)}(\mathcal{O}_X,\mathcal{E}(k)).\] Note that $\RR\omega_i$ is  well-defined from the Gorenstein condition.

\begin{definition}
Let $\calC$ be a triangulated category. $\calC=\langle \mathcal{A},\mathcal{B} \rangle$ is called a {\em semiorthogonal decomposition} of $\calC$ if
 \begin{enumerate}
  \item $\mathcal{A}$ and $\mathcal{B}$ are full triangulated subcategories.
  \item For any object in $C \in Ob(\calC)$, there is an exact triangle
   \[\xymatrix{A \ar[rr]^{[1]} & & B \ar[ld]\\ & C \ar[lu] & }\] for some $A \in Ob(\mathcal{A})$ and $B \in Ob(\mathcal{B})$. In this case, we call $A$ and $B$ as {\em orthogonal projections} of $C$ onto $\mathcal{A}$ and $\mathcal{B}$ respectively.
  \item $\Hom_\calC(B,A)=0$ for any $B \in Ob(\mathcal{B})$, $A \in Ob(\mathcal{A}).$
  
 \end{enumerate} 
\end{definition}

\begin{lemma}
Let $\mathcal{S}_{\geq i}$ be the triangulated subcategory generated by ${\bk}(e)$ with $e \leq -i$, and $\mathcal{P}_{\geq i}$ be the triangulated subcategory generated by $A(e)$ with $e \leq -i$.
Then for any $i\in \ZZ$ we have semiorthogonal decompositions
\[D^b({\rm gr-}A_{\geq i})=\langle \mathcal{D}_i, \mathcal{S}_{\geq i} \rangle = \langle \mathcal{P}_{\geq i},\mathcal{T}_i \rangle\]
where $\mathcal{D}_i$ is equivalent to $D^b({\rm qgr-}A)$ under the functor $\RR\omega_i$ and $\mathcal{T}_i$ is equivalent to $D^{gr}_{sg}(A)$.
\end{lemma}
\begin{remark}
Given an object in $D^b({\rm qgr-}A)$, we can directly construct an object in $D^b({\rm gr-}A_{\geq i})$ because the right adjoint functor is explicitly given as $\RR \omega_i$. On the other hand, it is more difficult to obtain an object in $D^b({\rm gr-}A_{\geq i})$ from an object in $D^{gr}_{sg}(A)$, because the left adjoint functor is not explicitly given. In the construction of the latter semiorthogonal decomposition of $D^b({\rm gr-}A_{\geq i})$, $\mathcal{P}_{\geq i}$ is first shown to be left admissible, and then $\mathcal{T}_i$ is just given by the left orthogonal of $\mathcal{P}_{\geq i}$. One can prove that $\mathcal{T}_i \cong D^b({\rm gr-}A_{\geq i})/\mathcal{P}_{\geq i}$ is equivalent to $D^{gr}_{sg}(A)$ by considering semiorthogonal decompositions of $D^b({\rm gr-}A)$ and $D^b({\rm grproj-}A)$.
\end{remark}

Now we introduce Orlov's main theorem.
\begin{theorem}[\cite{Or}]
Let $a$ be the Gorenstein parameter of $A$. Then $D^{gr}_{sg}(A)$ and $D^b({\rm qgr-}A)$ are related as following:
 \begin{enumerate}
  \item if $a>0$, for each $i \in \ZZ$ there are fully faithful functors $\Phi_i : D^{gr}_{sg}(A) \to D^b({\rm qgr-}A)$ and semiorthogonal decomposition
   \[D^b({\rm qgr-}A)=\langle \pi A(-i-a+1),...,\pi A(-i),\Phi_i D^{gr}_{sg}(A)\rangle\] where $\pi: D^b({\rm gr-}A) \to D^b({\rm qgr-}A)$ is the natural projection and 
   \[\Phi_i:D^{gr}_{sg}(A) \simeq \mathcal{T}_i \hookrightarrow D^b({\rm gr-}A) \stackrel{\pi}{\rightarrow} D^b({\rm qgr-}A).\]
  \item if $a<0$, for each $i \in \ZZ$ there are fully faithful functors $\CG_i: D^b({\rm qgr-}A) \to D^{gr}_{sg}(A)$ and semiorthogonal decomposition 
   \[D^{gr}_{sg}(A)=\langle q {\bf k}(-i),...,q {\bf k}(-i+a+1),\CG_i D^b({\rm qgr-}A)\rangle\] where $q: D^b({\rm gr-}A) \to D^{gr}_{sg}(A)$ is the natural projection and 
   \[\CG_i: D^b({\rm qgr-}A) \simeq \mathcal{D}_{i-a} \hookrightarrow D^b({\rm gr-}A) \stackrel{q}{\rightarrow} D^{gr}_{sg}(A).\]
  \item if $a=0$, $D^{gr}_{sg}(A)$ and $D^b({\rm qgr-}A)$ are equivalent via $\Phi_i$ and $\CG_i$.
 \end{enumerate}
\end{theorem}
If $X=\Proj(A)$ is a Calabi-Yau variety, i.e. $a$ is 0, then $D^b({\rm qgr-}A) (\cong D^b Coh(X))$ and $D^{gr}_{sg}(A) (\cong HMF_\ZZ(W))$ are equivalent. In particular, in this case $\mathcal{T}_i = \mathcal{D}_i$ in $D^b({\rm gr-}A_{\geq i})$.

The equivalence for the Calabi-Yau case is illustrated as follows. Suppose that we are given an object $\mathcal{E}$ in $D^b Coh(X)$. Then we get a complex $\RR \omega_i(\mathcal{E})\in D^b({\rm gr-}A_{\geq i}).$ We compute its minimal free resolution over $A$ such that the differentials give a matrix factorization of $W$. The other direction, namely from marix factorizations to coherent sheaves, is not as straightforward as before, because we do not have an explicit form of the functor from $HMF_\ZZ(W)(\simeq D^{gr}_{sg}(A))$ to $D^b({\rm gr-}A_{\geq i})$. So, given a matrix factorization $K^\cdot$, we take $M={\rm Cok}(k^{-1}) \in D^b({\rm gr-}A)=\langle D^b({\rm gr-}A_{\geq i}), \mathcal{P}_{<i} \rangle$, compute its orthogonal projection $M'$ onto $D^b({\rm gr-}A_{\geq i})=\langle \mathcal{P}_{\geq i},\mathcal{T}_i \rangle$, and compute its orthogonal projection $M''$ onto $\mathcal{T}_i$. Finally, we project $M'' \in D^b({\rm gr-}A_{\geq i}) \subset D^b({\rm gr-}A)$ to $D^b({\rm gr-}A)/D^b({\rm tors-}A)\simeq D^b({\rm qgr-}A)\simeq D^b Coh(X)$. Of course, instead of starting from ${\rm Cok}(k^{-1})$, we can use another $A$-module whose free resolution eventually gives $K^\cdot$.

We remark that the above correspondences are lifted to functors between dg-categories $D^b_\infty Coh(X)$ and $MF_\ZZ(W)$ by \cite{CT}. This fact will be used in the last section.




\section{Polishchuk-Zaslow mirror symmetry}\label{sec:CYCY}
Let $T^2$ be a symplectic torus given by $\CC/(\ZZ \oplus e^{2\pi i/3}\ZZ)$. Let $\alpha$ be its symplectic area, and $q:=T^\alpha$ in $\Lambda.$
Any unobstructed Lagrangian circle $L$ is described by a straight line $(c,0)+t\overrightarrow{v}$, $t\in \RR$, where $\overrightarrow{v}=a+e^{2\pi i/3}b$ with $a,b\in \ZZ$, and $c\in [0,1)$. In this case we denote $L=L_{(a,b),c}$, and if $c=0$ we omit it.

Then the categorical mirror symmetry of elliptic curves proved by \cite{PZ} is given by the following theorem:

\begin{theorem}[\cite{PZ}]
There is an equivalence of categories \[\CF: D^b Coh(X) \simeq D^\pi Fu(E),\] $\Phi(\CO_X(np_0))=L_{(1,-n)}$, $\Phi(\CO_{p_0})=L_{(0,-1),\frac{1}{2}}$ 
where $X$ is a mirror elliptic curve and $p_0$ is a point in $X$.
\end{theorem}

\begin{figure}
\includegraphics[height=2in]{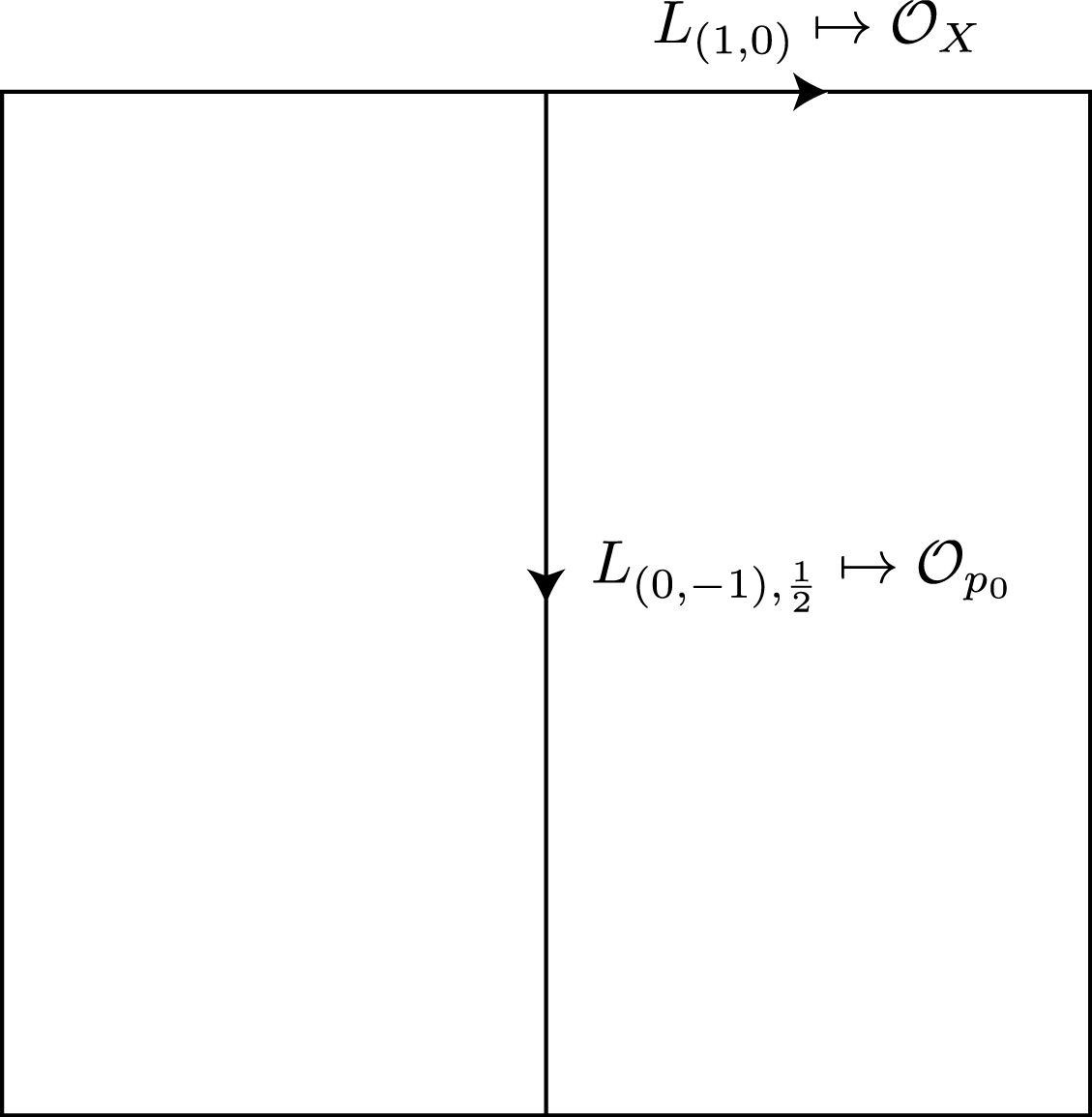}
\caption{Mirror correspondence by Polishchuk-Zaslow}
\end{figure}

The idea of the construction of $\CF$ in \cite{PZ} comes from theta identities. Indeed, theta functions can be understood as morphisms among line bundles on elliptic curves. Here we briefly illustrate their idea. They observe that theta identities given by compositions of morphisms between line bundles(i.e. products of theta functions) also naturally arise in the compositions of morphisms of the Fukaya category. To be more precise, first define a function $\theta[c]$ on $\Lambda^*$,
\[ \theta[c](w):=\sum_{m\in \ZZ}q^{(m+c)^2/2}w^{m+c}.\]
Then there is a degree 1 line bundle $\mathcal{L}$ on the elliptic curve such that its space of global sections is generated by $\theta[0](w)$, and $\mathcal{L}^n$ has global sections generated by $\theta[a/n](w^n)$ where $a \in \{0,1,...,n-1\}.$ The addition formula of theta functions is given by
\[ \theta[0](w)\cdot\theta[0](w)=\theta[0](1)\theta[0](w^2)+\theta[1/2](1)\theta[1/2](w^2).\]

On the other hand, let $\{p\}=L_{(1,0)} \cap L_{(1,-1)}=L_{(1,-1)}\cap L_{(1,-2)}$ and $\{q_0,q_1\}=L_{(1,0)}\cap L_{(1,-2)}$ where $p$ and $q_0$ are origin. $p$ is a morphism from $L_{(1,0)}$ to $L_{(1,-1)}$, or a morphism from $L_{(1,-1)}$ to $L_{(1,-2)}.$ Then 
\[m_2(p,p)=aq_0+bq_1\] is a morphism from $L_{(1,0)}$ to $L_{(1,-2)}$. $a$ and $b$ are given by counts of holomorphic triangles whose vertices are $p$, $p$, $q_0$ and $p$, $p$, $q_1$, respectively. The count is arranged with respect to the area, and it is easy to see that
\[a=2(q^{\frac{1}{2}(1\cdot 2)} + q^{\frac{1}{2}(2\cdot 4)}+ q^{\frac{1}{2}(3 \cdot 6)}+ \cdots)=2(q^1+q^4+q^9+ \cdots)=\sum_{m \in \ZZ}q^{m^2},\]
\[b=2(q^{\frac{1}{2}(\frac{1}{2}\cdot 1)}+q^{\frac{1}{2}(\frac{3}{2}\cdot 3)} + \cdots)=2(q^{1/4}+q^{9/4}+q^{25/4}+\cdots)=\sum_{m\in \ZZ} q^{(m+1/2)^2},\]
so $a=\theta[0](1)$ and $b=\theta[1/2](1).$ It implies that the mirrors of morphisms $p$, $q_0$ and $q_1$ among Lagrangian submanifolds are precisely theta functions which are natural basis of global sections of corresponding line bundles, or equivalently morphisms among them.


In \cite{AS} the mirror of $T^2$ is given by the Tate curve, while in \cite{CHL1} the mirror is $X=\Proj (R/W)$ where $R=\Lambda[x,y,z]$, $W=\phi(x^3+y^3+z^3)+\psi xyz$ for some $\phi,\psi \in \Lambda$ which will be defined later. 

We describe Polishchuk-Zaslow mirror correspondence between $T^2$ and $X$. Again, we take $\CO_X$ as the mirror of $L_{(1,0)}.$ Then the mirror of $L_{(1,-3)}$ is a line bundle of degree 3 whose global sections are generated by $\theta_0:=\theta[0](w^3)$, $\theta_1:=\theta[1/3](w^3)$ and $\theta_2:=\theta[2/3](w^3).$ They are mirrors of Lagrangian intersections of $L_{(1,0)}$ and $L_{(1,-3)}.$ On the other hand, considering $X$ as an abstract elliptic curve, they are used to define an embedding of $X$ into $\Lambda\PP^2$. Since we let $X$ as a projective cubic, the embedding is a priori given, so the mirrors of Lagrangian intersections are $x$, $y$ and $z$ which form a basis of global sections of $\CO_X(1).$ (From now on we write $\CO$ for $\CO_X$ if there is no confusion.) Therefore, we construct the mirror functor $\CF: D^\pi Fu(T^2) \to D^b Coh(X)$ by $\CF(L_{(1,0)}):=\CO$, $\CF(L_{(1,-3)}):=\CO(1)$ and the morphisms $(0,0),(1/3,0),(2/3,0) \in \Hom_{D^\pi Fu(T^2)}(L_{(1,0)},L_{(1,-3)})$ are mapped to $y$, $x$ and $z$ respectively.

\begin{remark}
As a divisor, from $\CO_{\Lambda\PP^2}(1)=(x=0)=(y=0)=(z=0)$, letting $\zeta:=e^{2\pi i/3}$,
\begin{eqnarray*}\CO_X(1) & \sim & [1:-1:0]+[1:-\zeta:0]+[1:-\zeta^2:0] \\
&\sim & [0:1:-1]+[0:1:-\zeta]+[0:1:-\zeta^2] \\
& \sim & [1:0:-1]+[1:0:-\zeta]+[1:0:-\zeta^2]
\end{eqnarray*}
and these nine points $[1:-\zeta^k:0]$, $[0:1:-\zeta^k]$, $[1:0:-\zeta^k]$ $(k=0,1,2)$ are called {\em inflection points} of $X$. Then for any inflection point $p$, $ \CO_X(3p) \cong \CO_X(1).$
\end{remark}
Now, we need an argument of Abouzaid-Smith \cite{AS} to make the above functor exact.
Instead of constructing the functor explicitly on arbitrary objects and morphisms, they compare $\AI$-subcategories of $Fu_0(E)$ and $D^b_\infty Coh(X)$ which consist of split-generators. Let $\Gamma\mathcal{A} \subset Fu_0(E)$ be an $\AI$-subcategory of Lagrangians $L_{(1,n)}$ for $n\in \ZZ$, and $\Gamma\mathcal{A}^\vee \subset D^b_\infty Coh(X)$ be a subcategory consisting of $\CO(np_0)$. They use Polishchuk's theorem on the $\AI$-structure of $\Gamma\mathcal{A}^\vee$: it is uniquely determined by its cohomology category and its lack of formality. By \cite{PZ}, $H(\Gamma\mathcal{A}) \simeq H(\Gamma\mathcal{A}^\vee).$ Then they prove that $\Gamma\mathcal{A}$ is also nonformal, so that there is an $\AI$-quasiequivalence between $\Gamma\mathcal{A}$ and $\Gamma\mathcal{A}^\vee$. Since the equivalence is between subcategories which consist of split-generators, it extends to an $\AI$-quasiequivalence between whole $\AI$-categories, and taking its cohomology we get an exact equivalence between triangulated categories.  




\section{CY-LG mirror symmmetry(graded localized mirror functors)}\label{sec:CYLG}
We review the construction of the localized mirror functor due to \cite{CHL1} for $T^2$ equipped with $\ZZ/3$-action, by rotation as in Figure \ref{333elliptic}. Let $\Omega=dz$ be the holomorphic volume form on $T^2$, where $z$ is the complex coordinate of $\mathbb{C}$. 

\begin{figure}
\includegraphics[height=3.5in]{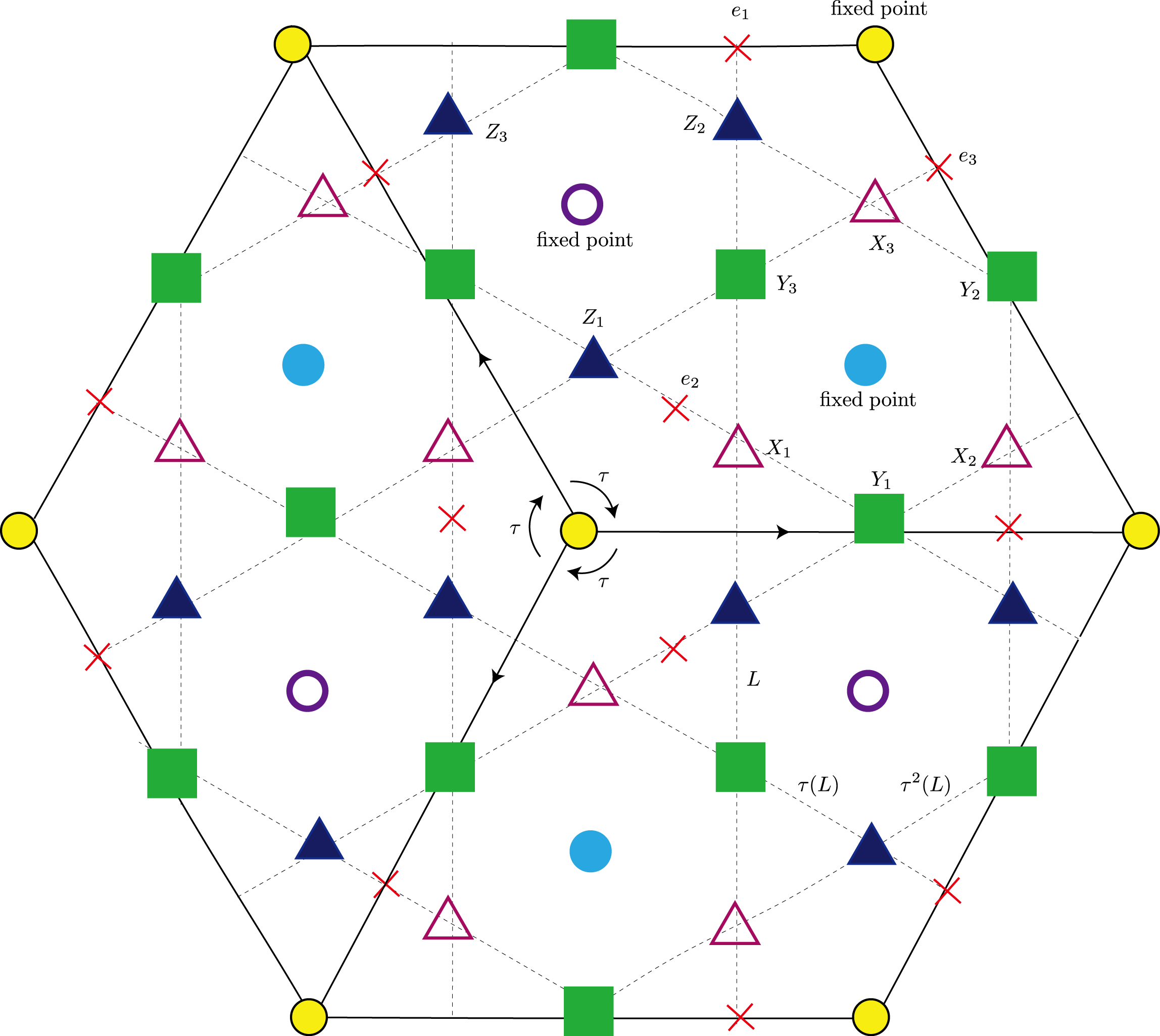}
\caption{$\ZZ/3$-action, $\LL$ on $E$ and self-intersection points of $\LL$. A parallelogram whose sides are solid lines is a fundamental domain of $T^2$ and the union of three dotted lines is $\LL$. $e_1$, $e_2$ and $e_3$ represent fundamental classes of $L$, $\tau(L)$ and $\tau^2(L)$ respectively, and $e=e_1+e_2+e_3$ is the unit of $\LL$.}
\label{333elliptic}
\end{figure}

If we consider grading structure on $E$, then the above $\ZZ/3$-action does not preserve it, so we do not have an induced action on the graded Fukaya category, but there is an action up to shifts, i.e. there is a homomorphism $\ZZ/3 \to Auteq(D^\pi Fu(T^2))/\ZZ$ where $\ZZ$ is generated by the shift functor $[1]$.

Let $\xi=e^{2\pi i/3}$, $L$ be an oriented Lagrangian given by the line $\frac{1+\xi}{2}+\sqrt{-1}t$ with orientation upward and phase $\frac{\pi}{2}$, i.e. $\displaystyle \frac{\Omega(V)}{|\Omega(V)|}=e^{i\cdot\frac{\pi}{2}}$ for any nonzero positively oriented vector field $V$ of $L$. We construct an immersed Lagrangian 
\[\LL:=L \cup \tau(L) \cup \tau^2(L)\]
and equip $L$, $\tau(L)$ and $\tau^2(L)$ with nontrivial spin structures. We define the phase of $\tau(L)$ as $\pi/2-2\pi/3$, and that of $\tau^2(L)$ as $\pi/2-4\pi/3$, i.e. $\tau$ is the rotation by $-2\pi/3$. For simplicity, we call $\LL$ the (lift of) Seidel Lagrangian.
\subsection{Ungraded localized mirror functor}
In this subsection we forget grading structures and just consider the $\ZZ/2$-grading. As in Figure \ref{333elliptic}, we specify three generators $e_1$, $e_2$ and $e_3$ of Morse complexes of $L$, $\tau(L)$ and $\tau^2(L)$ representing fundamental classes, and $e:=e_1+e_2+e_3.$ Then $e$ is the unit of $\LL$, and we also have the following important theorem.
\begin{theorem}[\cite{CHL1}]
Let $X$, $Y$ and $Z$ be immersed generators of $\bar{\LL} \subset T^2/(\ZZ/3)$, and $X_i$, $Y_i$ and $Z_i$ for $i=1,2,3$ be their liftings(see Figure \ref{333elliptic}). Then $b=x(X_1+X_2+X_3)+y(Y_1+Y_2+Y_3)+z(Z_1+Z_2+Z_3)$ for any $x,y,z\in \Lambda$ is a weak bounding cochain of $\LL$. Furthermore, $PO(\LL,b)= \phi(x^3-y^3+z^3)-\psi xyz$ where $\phi$ and $\psi$ are power series of $q_\alpha=T^{\Delta_{xyz}}$, such that $\Delta_{xyz}$ is the area of minimal $xyz$ triangle and 
\[ \phi(q_\alpha)=\sum_{k=0}^\infty (-1)^{k+1}(2k+1)q_\alpha^{(6k+3)^2},\]
\[ \psi(q_\alpha)=-q_\alpha+ \sum_{k=1}^\infty (-1)^{k+1}((6k+1)q_\alpha^{(6k+1)^2}-(6k-1)q_\alpha^{(6k-1)^2}).\]
\end{theorem}

Let $W:=PO(\LL,b).$ Suppose that $L'$ is an unobstructed Lagrangian submanifold. $CF((\LL,b),L')$ is generated by even and odd intersections(for example see Figure \ref{4x4mf1}). If we compute $m_1$ on it, i.e. count strips between even and odd generators, then $m_1^2=W\cdot \id_{n\times n}$ where $n$ is the number of odd(or even) generators, namely $(CF((\LL,b),L'),m_1)$ is a matrix factorization of $W$ of rank $n$. More precisely, we have
\begin{theorem}[\cite{CHL1}]
Let $\calC$ be an $\AI$-category and $(A,b)$ be a weakly unobstructed object with weak bounding cochain $b$. Let $\calC_0$ be the $\AI$-subcategory of unobstructed objects. Define a collection of maps $\{\locmir_*^{(A,b)}\}$ from $\calC_0$ such that
\begin{itemize}
 \item $\locmir_0^{(A,b)}$ sends an object $B$ to the matrix factorization $(hom((A,b),B),m_1)$.
 \item $\locmir_1^{(A,b)}(x)$ is defined as \[(-1)^* m_2^{b,0}(\cdot, x):(hom((A,b),B_1),m_1) \to (hom((A,b),B_2),m_1)\] where $x \in hom(B_1,B_2).$
 \item $\locmir_k^{(A,b)}(x_1,...,x_k)$ is defined by
 \[ \CF_k^{(A,b)}(x_1,...,x_k)(y)=(-1)^*m_{k+1}^{b,0,...,0}(y,x_1,...,x_k).\]
\end{itemize}
Then $\{\locmir_*^{(A,b)}\}: \calC_0 \to MF(PO(A,b))$ is an $\AI$-functor where $MF(PO(A,b))$ is the differential $\ZZ/2$-graded category of matrix factorizations of $PO(A,b)$.

\end{theorem}

We call the above functor $\{\locmir_*^{(A,b)}\}$ the (non-graded) {\em localized mirror functor} at $(A,b)$.

\subsection{Graded localized mirror functor}
Now we explain the construction of \cite{CHL2} in the case of cyclic group action (while
the construction there of is for any finite group action). We also refer to the Chapter 5 of \cite{CHL1} for more details. Let $(M,\omega,\Omega)$ be a Calabi-Yau manifold where $\Omega$ is the holomorphic volume form, and suppose $\ZZ/d\ZZ$ acts on $M$. Consider {\em $\frac{1}{d}$-grading} on an immersed Lagrangian, i.e. define a phase function $\theta_L^{1/d}: L \to \RR$ such that $\Omega^{\otimes d}(T_p L)=e^{\pi i \theta_L^{1/d}(p)}.$ If there is such a phase function, then $L$ is called {\em $\frac{1}{d}$-graded.} Furthermore if it can be equipped with a constant phase function, then it is called {\em special Lagrangian} with respect to the $\frac{1}{d}$-grading.

\begin{prop}[\cite{CHL2}]
On the full $\AI$-subcategory of $\frac{1}{d}$-graded Lagrangians, the $\AI$-multiplication has degree $2-k$ under $\frac{1}{d}$-grading.
\end{prop}

Under $\frac{1}{d}$-grading, the degree of an intersection $p$ between $\frac{1}{d}$-graded Lagrangians is defined as 
\begin{equation}\label{fracgrading}
\deg^{1/d}(p):=\frac{1}{d}(\theta_{L_2}^{1/d}(p)-\theta_{L_1}^{1/d}(p)+\theta^d(\wideparen{L_2 L_1})|_p) \in \frac{1}{d}\ZZ
\end{equation}
where $p$ is considered as a morphism $L_1 \to L_2$ and $\pi \theta^d(\wideparen{L_2 L_1})|p$ is the phase angle of the positive definite path from $T_p L_2$ to $T_p L_1$ measured by $\Omega^{\otimes d}(p).$ 

We go back to our example $T^2$ on which $\ZZ/3$ acts. The Seidel Lagrangian $\LL$ can be made into a special Lagrangian with respect to $\frac{1}{3}$-grading, defining the phase $\theta_\LL^{1/3}=-\frac{1}{2},$ for example. Then each immersed point which is considered to be a morphism $\tau^2(L) \to \tau(L)$, $\tau(L) \to L$ or $L \to \tau^2(L)$(i.e. an odd-degree morphism in $\ZZ/2$-grading) has degree $1/3$ according to the definition (\ref{fracgrading}).

To make $b=x\sum_{i=1}^3 X_i + y\sum_{i=1}^3 Y_i + z\sum_{i=1}^3 Z_i$ of degree 1, let $\deg^{1/3}(x)=\deg^{1/3}(y)=\deg^{1/3}(z)=2/3$. Suppose that an unobstructed(or $\frac{1}{3}$-graded) Lagrangian $L'$ is given. The morphisms $\LL \to L'$ are equipped with $\frac{1}{3}$-gradings, and since $b$ has degree 1, $W=PO(\LL,b)$ is of degree 2. $CF((\LL,b),L')$ is equipped with $m_1$ of degree 1 and $m_1^2=W \cdot \id_{n\times n}$ for some $n$.

Now fix a labelling on components of $\LL$ by elements of $\ZZ/3$ according to the group action. More precisely, for example fix $L^0=L$ and according to the $\ZZ/3$-action $L^{-1}=\tau(L)$ and $L^{-2}=\tau^2(L)$(here 0, $-1$ and $-2$ are elements of $\ZZ/3$). Then fix a character $\ZZ/3 \to U(1)$, $-j \mapsto e^{-\pi i \cdot 2j/3}$.
Take an unobstructed Lagrangian $L'$ again, and define a pair
\[\displaystyle\big(\bigoplus_{p_i \in CF((\LL,b),L)} A_{g_i}[\deg p_i],m_1^{(\LL,b),L}\big),\] $p_i \in L^{g_i} \cap L'$ and $\deg p_i$ is the usual $\ZZ$-grading by $\Omega$.  

We explain the expressions above. Given a polynomial ring $R$ and a (quasi-)homogeneous polynomial $W$ of degree $d$, we define a category $Tw_\ZZ(R_W \# \ZZ/d)$ which consists of objects as pairs $(\bigoplus_{g_i \in \ZZ/d}A_{g_i}[\sigma_i],\delta)$. $ \Hom(A_{g},A_h)$ is a vector space consisting of $f$ which is required to satisfy \[\widetilde{\deg}f:=\deg^{1/d}f+(\alpha_h-\alpha_g) \in 2\ZZ.\] Here we fix a character $g \mapsto e^{\pi i \alpha_g}$.
We give a $\ZZ$-grading on $\Hom$ spaces by $\widetilde{{\rm deg}}$. $\delta$ is a degree 1 endomorphism of $\bigoplus_{g_i \in \ZZ/3}A_{g_i}[\sigma_i]$ where the degree of a morphism $A_g[\sigma]\to A_h[\sigma']$ is shifted by $\sigma'-\sigma$ from the degree of the morphism $A_g \to A_h$. The upshot is the following:

\begin{theorem}[\cite{CT}]\label{CTtheorem}
$Tw_\ZZ(R_W \# \ZZ/d)$ is equivalent to the category of graded matrix factorizations of $W$ by the correspondence:
\[ \big(\bigoplus_i A_{g_i}[k_i],\delta\big) \mapsto (\cdots \to E_0 \to E_1 \to \cdots)\]
where \[ E_0=\bigoplus_{k_i:{\rm even}}R\big(-\frac{k_i d}{2}-\frac{d}{2}\alpha_{g_i}\big),\]
\[ E_1=\bigoplus_{k_i:{\rm odd}} R\big(\frac{(1-k_i)d}{2}-\frac{d}{2}\alpha_{g_i}\big)\]
and the structure maps $p_i: E_i \to E_{i+1}$ are given by the corresponding matrix defined by $\delta$.
\end{theorem}

\begin{prop}[\cite{CHL2}]\label{CHL2theorem}
The above pair $\displaystyle\big(\bigoplus_{p_i \in CF((\LL,b),L)} A_{g_i}[\deg p_i],m_1^{(\LL,b),L}\big)$ is an object of $Tw_\ZZ(R_W\#\ZZ/3)$ where $W=PO(\LL,b)\in R=\Lambda[x,y,z]$, i.e. it gives a graded matrix factorization of $W$ with the usual grading on $R$, i.e. $\deg(x)=\deg(y)=\deg(z)=1$. Furthermore, the collection of maps $\{\CF_*^{\LL,b}\}$ becomes an $\AI$-functor $Fu^\ZZ(T^2) \to Tw_\ZZ(R_W\#\ZZ/3).$
\end{prop}

The above $\AI$-functor in the proposition is called the {\em graded localized mirror functor} and denoted by $\locmir_{gr}^{\LL,b}$. We often omit $b$ from the notation. We remark that we followed the convention of \cite{CHL2} which is different from that of \cite{CT}. Also we mention that in \cite{CHL2} Proposition \ref{CHL2theorem} is proved in much more general setting, as {\em noncommutative} matrix factorizations. 

\section{Main theorem}\label{sec:mainthm}
Now we are ready to state and prove our main theorem. 
Recall the notation $\CG_i=q_i \circ \RR\omega_i: D^b(qgr-A) \to D^{gr}_{sg}(A)$ from Section 4.

\begin{theorem}
Suppose that $S_0=t_{(0,1/2)}\circ \left(\begin{array}{cc}1 & 0 \\2 & 1\end{array}\right)$ be a symplectomorphism of $E$, where $t_{(0,1/2)}$ is the translation by $(0,1/2)$. Let $\CS_0: Fu(T^2) \to Fu(T^2)$ be an autoequivalence induced by $S_0$. Then we can construct an equivalence $\widetilde{\CF}:D^\pi Fu(T^2) \to D^b Coh(X)$ which equals to Polishchuk-Zaslow's functor on a split-generating subcategory $\calA \subset D^\pi Fu(T^2)$ such that $\CG_0 \circ \widetilde{\CF}$ can be identified by $\locmir^\LL_{gr} \circ \CS_0$.
\end{theorem}
This section is devoted to the proof of this theorem. Pick two Lagrangians $L_0:=L_{(1,0)}$ and $L_1:=L_{(1,-3)}$ in $D^\pi Fu_0(E)$, both equipped with nontrivial spin structures. Let $\calA$ be the $\AI$-subcategory consisting of $L_0$ and $L_1$. Let $\CF$ be the Polishchuk-Zaslow's mirror functor. Then $L_0$ and $L_1$ are mapped to $\CO$ and $\CO(1)$ via $\CF$ respectively. Observe that the Lagrangians which we picked generate $D^\pi Fu_0(E)$ and their images via $\CF$ also generate $D^b Coh(X)$. 

\subsection{Computations via $\CG_0 \circ \CF$}
We need to compute image objects of $\CO$, $\CO(1)$ and morphisms between them via $\CG_0$. 

\[\RR \omega_0(\CO)=\bigoplus_{k=0}^\infty \RR\Hom_{D^b Coh(X)}(\CO,\CO(k))\in D^b({\rm gr-}A_{\geq 0}) \hookrightarrow D^b({\rm gr-}A),\]
\[\RR \omega_0(\CO(1))=\bigoplus_{k=0}^\infty \RR\Hom_{D^b Coh(X)}(\CO,\CO(1)(k))\in D^b({\rm gr-}A_{\geq 0}) \hookrightarrow D^b({\rm gr-}A).\]

We recall a useful lemma of homological algebra. 

\begin{lemma} 
For a chain complex $C^\cdot$, if $H^i(C^\cdot) \cong M$ and $H^j(C^\cdot)=0$ for $j\neq i$, then $C^\cdot$ is quasi-isomorphic to $M[-i]$. 
\end{lemma}
A simple computation of Ext groups gives $\RR^0 \omega_0(\CO(1)) \cong A(1)_{\geq 0}$, and $\RR^1 \omega_0(\CO(1)) = 0.$ So by the lemma we have a quasi-isomorphism of complexes
\[\RR \omega_0(\CO(1)) \simeq A(1)_{\geq 0}\] (the right hand side is considered to be a complex concentrated in degree 0). Its minimal $A$-free resolution gives rise to the corresponding matrix factorization via $\CG_0$.

It is also easy to see that $\RR^0 \omega_0(\CO) \cong A$ and $\RR^1 \omega_0(\CO) \cong \Lambda$.
The complex $\CE^\cdot:=\RR \omega_0(\CO)$ itself can be explicitly obtained $\CE^\cdot$ by the following lemma.

\begin{lemma}
The complex $\CE^\cdot$ fits into an exact triangle
\[A \to \CE^\cdot \to \bk[-1] \to A[1]\] in $D^b ({\rm gr-}A)$, where $\bk[-1] \to A[1]$ is a nonzero morphism.
\end{lemma}

\begin{proof}[Proof of the lemma.]
We follow the method of \cite{Asp}. Recall that $A$ is a Gorenstein algebra satisfying $\RR \Hom_A(\bk,A)=\bk[-2].$ 
Hence, \[\Ext^2_A(\bk,A) \cong \Lambda,{\rm \;or\;\;} \Hom_{D^b({\rm gr-}A)}(\bk,A[2])=\Hom_{D^b({\rm gr-}A)}(\bk[-1],A[1]) \cong \Lambda\] and $\Ext^i_A(\bk,A)=0$ if $i\neq 2$.

Pick any nonzero morphism $f:\bk[-1] \to A[1]$ and let $C$ be its cocone, i.e. 
\begin{equation}\label{extri}
\xymatrix{A \ar[r] & C \ar[r] & \bk[-1]\ar[r]^f & A[1]}
\end{equation}
 is an exact triangle.
 
Then applying $\Hom(\cdot,A(r))$ for $r \leq 0$ in $D^b({\rm gr-}A)$ gives a long exact sequence
\[\xymatrixrowsep{0.2pc}\xymatrix{\cdots \ar[r]& \Hom(\bk[-1],A(r)) \ar[r]& \Hom(C,A(r)) \ar[r]& \Hom(A,A(r)) \\ 
\ar[r] & \Hom(\bk[-2],A(r)) \ar[r] & \Hom(C[-1],A(r)) \ar[r]& \Hom(A[-1],A(r)) \ar[r]& \cdots.}\]

Let $r=0$. Then a part of the above sequence is given by
\[\xymatrix{0 \ar[r] & \Hom(C,A) \ar[r] & \Lambda \ar[r]^{f[-1]^*} & \Lambda \ar[r] & \Hom(C[-1],A)\ar[r] & \Hom(A[-1],A)=0.}\]
Since $f$ is nonzero, $f[-1]^*$ is an injective linear map from $\Lambda$ to itself. So it is also surjective, and $\Hom(C,A)=\Hom(C[-1],A)=0$.
If $i\neq 0,-1$, then the exact sequence
\[\xymatrix{0=\Hom(\bk[i-1],A) \ar[r] & \Hom(C[i],A) \ar[r] & \Hom(A[i],A)=0}\] gives $\Hom(C[i],A)=0.$ Hence $\Hom(C[i],A)=0$ for all $i \in \ZZ.$
Now let $r<0.$ Then $\Hom(\bk[i],A(r))=0$ for all $i \in \ZZ$ 
by Gorenstein condition. 
Clearly $\Hom(A[i],A(r))=0$ for any $i \in \ZZ$ and $r<0$. Consequently $\Hom(C[i],A(r))=0$ for any $i\in \ZZ$, $r \leq 0.$ By the semiorthogonal decomposition $D^b({\rm gr-}A_{\geq 0})=\langle \mathcal{P}_{\geq 0},\mathcal{T}_0 \rangle,$ since $\Hom(C,P)=0$ for all $P \in \mathcal{P}_{\geq 0}$, $C$ is in $\mathcal{T}_0.$ 

Finally, via $\pi: D^b({\rm gr-}A) \to D^b({\rm qgr-}A)$, by (\ref{extri}), $\pi C$ is equivalent to $\pi A$ which corresponds to $\CO \in D^b Coh(X)$. Since $\RR \omega_0:D^b Coh(X) \to \mathcal{D}_0(=\mathcal{T}_0)$ is an equivalence, $C$ is isomorphic to $\CE^\cdot=\RR \omega_0(\CO).$
\end{proof}

By the lemma, $\CE^\cdot$ and the mapping cone of $g: \bk[-2] \to A$ are quasi-isomorphic as chain complexes. Since $\bk\cong A/(x,y,z)$ and $(x,y,z)$ is a regular sequence of $R=\Lambda[x,y,z]$, we follow the algorithm in \cite{Dyc} to take the free resolution of $\bk$. It is given by a double complex
\[\xymatrixrowsep{4pc}\xymatrixcolsep{4pc}
\xymatrix{
  \vdots \ar[d] & \vdots\ar[d]& \vdots\ar[d]& & \\
 A(-8)^3 \ar[r] \ar[d] & A(-7)^3\ar[r]\ar[d] & A(-6)\ar[d] & & & \\
 A(-6)\ar[r]  & A(-5)^3 \ar[r] \ar[d]_{\left(\begin{smallmatrix}w_x &w_y & w_z\end{smallmatrix}\right)} & A(-4)^3 \ar[r] \ar[d]_{\tiny\left(\begin{smallmatrix}0 & \frac{w_z}{\alpha} & -\frac{w_y}{\alpha} \\-\frac{w_z}{\alpha} & 0 & \frac{w_x}{\alpha} \\ \frac{w_y}{\alpha} & -\frac{w_x}{\alpha} & 0\end{smallmatrix}\right)}
& A(-3) \ar[d]_{\left(\begin{smallmatrix} w_x \\ w_y \\ w_z\end{smallmatrix}\right)} 
& & \\
  & A(-3) \ar[r]_{\left(\begin{smallmatrix}x \\y \\z\end{smallmatrix}\right)}& A(-2)^3 \ar[r]_{\tiny\left(\begin{smallmatrix}0 & -\alpha z & \alpha y \\\alpha z & 0 & -\alpha x \\-\alpha y & \alpha x & 0\end{smallmatrix}\right)} & A(-1)^3 \ar[r]_-{\left(\begin{smallmatrix}x & y & z\end{smallmatrix}\right)} & A\ar[r] & 0}\]
where $\displaystyle\alpha=\sum_{k=0}^\infty \big((-1)^k T^{et(1+6k)}+(-1)^{k+1} T^{et(5+6k)}\big)$, $W=xw_x+yw_y+zw_z$ with
\[w_x=x^2\sum_{k=0}^\infty (-1)^{k+1}(2k+1)q_\alpha^{(6k+3)^2}+yz\sum_{k=1}^\infty (-1)^{k+1}(2kq_\alpha^{(6k+1)^2}-2kq_\alpha^{(6k-1)^2}),\]
\[w_y=y^2\sum_{k=0}^\infty (-1)^{k}(2k+1)q_\alpha^{(6k+3)^2}+zx\sum_{k=1}^\infty (-1)^{k}(2kq_\alpha^{(6k+1)^2}-2kq_\alpha^{(6k-1)^2}),\]
\[w_z=z^2\sum_{k=0}^\infty (-1)^{k+1}(2k+1)q_\alpha^{(6k+3)^2}+xy\sum_{k=1}^\infty (-1)^{k+1} ((2k+1)q_\alpha^{(6k+1)^2}-(2k-1)q_\alpha^{(6k-1)^2})-xyq_\alpha.\] Here $et(a)$ means the area of an equilateral triangle of face length $a$(we let the length of the minimal triangle be 1). $\alpha$ equals to $\gamma$ which is in Definition 7.8 of \cite{CHL1}. We choose this specific free resolution to make the comparison of objects more easily.

The maps which are not specified are just copies of written ones. Observe that each row is given by the usual Koszul complex of $(x,y,z)$, and the vertical maps are needed to capture relations caused by $W$. It is clear that the complex is eventually quasi-2-periodic and gives a matrix factorization of $W$. Now the free resolution of $\bk[-2]$ is written as
\[\xymatrixrowsep{0.2pc}\xymatrixcolsep{3pc}\xymatrix{
 \cdots \ar[r] & A(-6)\oplus A(-5)^3\ar[rr]^{\tiny\left(\begin{smallmatrix} 0 & w_x & w_y & w_z \\ w_x &  0 & -\alpha z & \alpha y \\ w_y & \alpha z & 0 & -\alpha x \\ w_z & -\alpha y & \alpha x & 0\end{smallmatrix}\right)} 
 & &A(-3) \oplus A(-4)^3 \ar[rr]^-{\tiny\left(\begin{smallmatrix}0 & x & y & z \\x & 0 & \frac{w_z}{\alpha} & -\frac{w_y}{\alpha} 
 \\y & -\frac{w_z}{\alpha} & 0 & \frac{w_x}{\alpha} \\z & \frac{w_y}{\alpha} & -\frac{w_x}{\alpha} & 0\end{smallmatrix}\right)} & & A(-3)\oplus A(-2)^3
 \\ & & & & & \fbox{{\rm 0th}} } 
\]
\[\xymatrixrowsep{0.2pc}\xymatrixcolsep{3pc}\xymatrix{
 \ar[rr]^-{\tiny\left(\begin{smallmatrix} w_x &  0 & -\alpha z & \alpha y \\ w_y & \alpha z & 0 & -\alpha x \\ w_z & -\alpha y & \alpha x & 0\end{smallmatrix}\right)} & & A(-1)^3 \ar[r]^-{\left(\begin{smallmatrix} x&y&z\end{smallmatrix}\right)} & A \ar[r] & 0. & & & & \\ & & \fbox{{\rm 1st}}&\fbox{{\rm 2nd}} & &}
\]
Then any map $\bk[-2] \to A$ is given by $\phi: A(-3)\oplus A(-2)^3 \to A$ such that 
\[\phi\circ \left(\begin{matrix}0 & x & y & z \\x & 0 & \frac{w_z}{\alpha} & -\frac{w_y}{\alpha} 
 \\y & -\frac{w_z}{\alpha} & 0 & \frac{w_x}{\alpha} \\z & \frac{w_y}{\alpha} & -\frac{w_x}{\alpha} & 0\end{matrix}\right) =0\]
and it is easy to see that $\phi=\left(\begin{matrix}0 & w_x & w_y & w_z\end{matrix}\right)$ gives a nontrivial morphism. It is nothing but the first row of the consecutive differential map of the resolution. Any other row also defines a chain map, but then it becomes homotopically trivial.

So the mapping cone $C(\phi)$, which is isomorphic to $\RR \omega_0(\CO)$, is given by 
\[\xymatrixrowsep{0.2pc}\xymatrixcolsep{2.6pc}\xymatrix{
\cdots \ar[rr]^-{\tiny\left(\begin{smallmatrix}0 & x & y & z \\x & 0 & \frac{w_z}{\alpha} & -\frac{w_y}{\alpha} 
 \\y & -\frac{w_z}{\alpha} & 0 & \frac{w_x}{\alpha} \\z & \frac{w_y}{\alpha} & -\frac{w_x}{\alpha} & 0\end{smallmatrix}\right)} & & A(-3)\oplus A(-2)^3 \ar[rr]^{\tiny\left(\begin{smallmatrix} 0 & w_x & w_y & w_z \\ w_x &  0 & -\alpha z & \alpha y \\ w_y & \alpha z & 0 & -\alpha x \\ w_z & -\alpha y & \alpha x & 0\end{smallmatrix}\right)} 
 & &
 A \oplus A(-1)^3 \ar[r]^-{\tiny\left(\begin{smallmatrix} 0&x&y&z\end{smallmatrix}\right)}
 &A \ar[r] & 0. \\
 & & & & & \fbox{{\rm 0th}} & \fbox{{\rm 1st}} &
 }\]

Now we compute the free resolution of $\RR \omega_0(\CO(1)) \simeq A(1)_{\geq 0}$. Since $A(1)_{\geq 0}$ is generated by $x,y$ and $z$, the resolution starts from
\[\xymatrixcolsep{2.6pc}\xymatrix{\cdots \ar[r] & F^{-1} \ar[r] & A^3 \ar[r]^-{(x\;\;y\;\;z)} & A(1)_{\geq 0} \ar[r] & 0}\]
and the same argument as above gives the following free resolution of $A(1)_{\geq 0}$:
\[\xymatrixcolsep{2.6pc}\xymatrix{
\cdots \ar[r] & A(-2) \oplus A(-3)^3 \ar[rr]^-{\tiny\left(\begin{smallmatrix}0 & x & y & z \\x & 0 & \frac{w_z}{\alpha} & -\frac{w_y}{\alpha} 
 \\y & -\frac{w_z}{\alpha} & 0 & \frac{w_x}{\alpha} \\z & \frac{w_y}{\alpha} & -\frac{w_x}{\alpha} & 0\end{smallmatrix}\right)} & & A(-2)\oplus A(-1)^3 \ar[rr]^-{\tiny\left(\begin{smallmatrix} w_x &  0 & -\alpha z & \alpha y \\ w_y & \alpha z & 0 & -\alpha x \\ w_z & -\alpha y & \alpha x & 0\end{smallmatrix}\right)} & & A^3 \ar[r]
 & 0.}\]

Now we are ready to compare morphisms, $HF^k(L_i,L_j)$ and $\Hom_{HMF_\ZZ(W)}(M_i,M_j[k]),$ $i,j=$ 0 or 1 and $k=$ 0 or 1. Here $M_i$ and $M_j$ are matrix factorizations corresponding to $L_i$ and $L_j$ via above correspondence.
$HF^0(L_0,L_0)$ and $HF^0(L_1,L_1)$ are generated by identity morphisms, and any functor preserves identities, so we do not need any computation for degree 0 endomorphisms.

Recall that three intersections of $L_{(1,0)}$ and $L_{(1,-3)}$, which are basis of $HF^0(L_0,L_1)$, correspond to $\CO \stackrel{x}{\longrightarrow} \CO(1)$, $\CO \stackrel{y}{\longrightarrow} \CO(1)$ and $\CO \stackrel{z}{\longrightarrow} \CO(1)$ via $\CF$. We need to know how they correspond to morphisms between matrix factorizations via $\CG_0$. We compute the example $\CO \stackrel{x}{\longrightarrow} \CO(1)$. $f:=\RR \omega_0(x): \RR \omega_0(\CO) \to \RR \omega_0(\CO(1))$ is a map 

\[\xymatrixrowsep{4pc}\xymatrixcolsep{2.6pc}\xymatrix{
\cdots\ar[rr]^-{\tiny\left(\begin{smallmatrix}0 & x & y & z \\x & 0 & \frac{w_z}{\alpha} & -\frac{w_y}{\alpha} 
 \\y & -\frac{w_z}{\alpha} & 0 & \frac{w_x}{\alpha} \\z & \frac{w_y}{\alpha} & -\frac{w_x}{\alpha} & 0\end{smallmatrix}\right)} & & A(-3)\oplus A(-2)^3 \ar[d]^{f^{-1}}\ar[rr]^-{\tiny\left(\begin{smallmatrix} 0 & w_x & w_y & w_z \\ w_x &  0 & -\alpha z & \alpha y \\ w_y & \alpha z & 0 & -\alpha x \\ w_z & -\alpha y & \alpha x & 0\end{smallmatrix}\right)} 
 & &
 A \oplus A(-1)^3 \ar[d]^{f^0}\ar[r]^-{\tiny\left(\begin{smallmatrix} 0&x&y&z\end{smallmatrix}\right)}
 &A \ar[d]\ar[r] & 0 \\
 \cdots \ar[rr]^-{\tiny\left(\begin{smallmatrix}0 & x & y & z \\x & 0 & \frac{w_z}{\alpha} & -\frac{w_y}{\alpha} 
 \\y & -\frac{w_z}{\alpha} & 0 & \frac{w_x}{\alpha} \\z & \frac{w_y}{\alpha} & -\frac{w_x}{\alpha} & 0\end{smallmatrix}\right)} & & A(-2)\oplus A(-1)^3 \ar[rr]^-{\tiny\left(\begin{smallmatrix} w_x &  0 & -\alpha z & \alpha y \\ w_y & \alpha z & 0 & -\alpha x \\ w_z & -\alpha y & \alpha x & 0\end{smallmatrix}\right)} & & A^3 \ar[r] & 0 \ar[r] & 0} 
 \]
which induces $x: A \to A(1)_{\geq 0}$, namely the map $x$ between $\RR^0\omega_0(\CO)(\cong A)$ and $\RR^0\omega_0(\CO(1))(\cong A(1)_{\geq 0})$. It turns out that maps between the 0th cohomologies completely determine the maps between genuine complexes, because
\begin{eqnarray*} 
\dim_\Lambda\Hom_{D^b({\rm gr-}A)}(\RR\omega_0(\CO),\RR\omega_0(\CO(1)))&=&\dim_\Lambda\Hom_{D^bCoh(X)}(\CO,\CO(1)) \\
&=&3 \\
&=&\dim_\Lambda\Hom_{{\rm gr-}A}(A,A(1)_{\geq 0})
\end{eqnarray*}
(the first identity comes from the fact that $\RR\omega_0$ is fully faithful). Therefore, instead of trying to compute the morphism $f$ completely, we just try to describe data of $f$ which are sufficient to determine it.

From $\CO \stackrel{x}{\longrightarrow} \CO(1)$, the induced morphism of complexes is determined by $H^0(f)$, and it induces $A \stackrel{x}{\longrightarrow} A(1)_{\geq 0}$ if \[f^0=\left(\begin{array}{cccc}1 & * & * & * \\ 0 & * & * & * \\ 0 & * & * & *\end{array}\right).\] Similarly, if we start from $\CO \stackrel{y}{\longrightarrow} \CO(1)$ or $\CO \stackrel{z}{\longrightarrow} \CO(1)$, then the induced maps are \[g^0= \left(\begin{array}{cccc}0 & * & * & * \\ 1 & * & * & * \\ 0 & * & * & *\end{array}\right) {\rm \;\;or\;\;} h^0=\left(\begin{array}{cccc}0 & * & * & * \\ 0 & * & * & * \\ 1 & * & * & *\end{array}\right)\] respectively. 

We also need to examine the correspondence of higher degree morphisms, i.e. we compare $HF^1(L_i,L_j)\cong HF^0(L_i,L_j[1])$ and $\Hom_{HMF_\ZZ(W)}(M_i,M_j[1]).$ The Serre duality for $X$ gives $\Ext^1(\CO(1),\CO) \cong \Ext^0(\CO,\CO(1))^*$ and $\Ext^1(\CO,\CO(1)) \cong \Ext^0(\CO(1),\CO)^*=0.$ We describe the morphism $\CO(1) \stackrel{x^*}{\longrightarrow} \CO[1]$ as a morphism between $\RR \omega_0(\CO(1))$ and $\RR \omega_0(\CO[1]),$ where $x^*\in \Ext^1(\CO(1),\CO)$ is the dual of $\CO \stackrel{x}{\longrightarrow} \CO(1)$. $g:=\RR\omega_0(x^*): \RR\omega_0(\CO(1)) \to \RR\omega_0(\CO[1])$ is a map

\begin{gather}
\xymatrixrowsep{3pc}\xymatrixcolsep{2.6pc}\xymatrix{
\cdots \ar[r] & A(-2) \oplus A(-3)^3 \ar[d]^{f'^{-2}}\ar[rr]^-{\tiny\left(\begin{smallmatrix}0 & x & y & z \\x & 0 & \frac{w_z}{\alpha} & -\frac{w_y}{\alpha} 
 \\y & -\frac{w_z}{\alpha} & 0 & \frac{w_x}{\alpha} \\z & \frac{w_y}{\alpha} & -\frac{w_x}{\alpha} & 0\end{smallmatrix}\right)} & & A(-2)\oplus A(-1)^3 \ar[d]^{f'^{-1}}\ar[rr]^-{\tiny\left(\begin{smallmatrix} w_x &  0 & -\alpha z & \alpha y \\ w_y & \alpha z & 0 & -\alpha x \\ w_z & -\alpha y & \alpha x & 0\end{smallmatrix}\right)} & & A^3 \ar[d]^{f'^0} \ar[r] & 0 \\
\cdots\ar[r] & A(-3)\oplus A(-2)^3 \ar[rr]^-{\tiny\left(\begin{smallmatrix} 0 & w_x & w_y & w_z \\ w_x &  0 & -\alpha z & \alpha y \\ w_y & \alpha z & 0 & -\alpha x \\ w_z & -\alpha y & \alpha x & 0\end{smallmatrix}\right)} 
 & &
 A \oplus A(-1)^3 \ar[rr]^-{\left(\begin{smallmatrix} 0&x&y&z\end{smallmatrix}\right)}
 & & A \ar[r] & 0.}
 \label{deg1mfmorph}
 \end{gather}
It is also determined by the map of the 0th cohomologies by dimension arguments as above. Recall that $\RR^0\omega_0(\CO(1)) \cong A(1)_{\geq 0}$ and it is generated by $x$, $y$ and $z$ which are identified as morphisms from $\CO$ to $\CO(1)$. On the other hand, $\RR^0\omega_0(\CO[1]) \cong A/(x,y,z) \cong \Lambda$, and $\Hom_{{\rm gr-}A}(A(1)_{\geq 0},A/(x,y,z))$ has basis $\{\phi_x,\phi_y,\phi_z\}$ where $\phi(x)$ is defined by \[\phi_x(x)=1, \phi_x(y)=\phi_x(z)=0,\] $\phi_y$ and $\phi_z$ are defined similarly. Hence $x^*$ corresponds to $\phi_x$, which is induced by $g^0=(1 \;\; 0\;\; 0):A^3 \to A$. Similarly $y^*$ corresponds to $(0\;\;1\;\;0):A^3 \to A$, and $z^*$ to $(0\;\;0\;\;1):A^3 \to A.$ As before, they completely determine $g^{-1}$, $g^{-2}$,$\cdots$, so give rise to a morphism of matrix factorizations. Finally, under the duality $\Ext^0(\CO,\CO) \cong \Ext^1(\CO,\CO)^*$ and $\Ext^0(\CO(1),\CO(1)) \cong \Ext^1(\CO(1),\CO(1))^*$, 
\[x^* \circ x = y^* \circ y= z^* \circ z = (\id_\CO)^* \in \Ext^1(\CO,\CO),\] 
\[x \circ x^*=y\circ y^*= z\circ z^* = (\id_{\CO(1)})^* \in \Ext^1(\CO(1),\CO(1))\]
and they are mirrors of the basis of $HF^1(L_0,L_0)$ and $HF^1(L_1,L_1)$ respectively. There are corresponding morphisms of matrix factorizations given by compositions of maps computed above. Finally, replacing all free $A$-modules by $R$-modules, i.e. considering above chain complexes over $A$ as matrix factorizations of $W$, and extending it (quasi)2-periodically, we get objects and morphisms in $HMF_\ZZ(W).$

\subsection{Computations via $\locmir_{gr}^\LL \circ \CS_0$} Via $\CS_0$, $L_0$ is mapped to $L$ which is a branch of $\LL$ and $L_1$ is mapped to $\tau(L)$ which is another branch of $\LL.$ Write $L_0':=\CS_0(L_0)=L$ and $L_1':=\CS_0(L_1)=\tau(L).$  Corresponding graded matrix factorizations $M_0=\locmir_{gr}^\LL(L_0')$ and $M_1:=\locmir_{gr}^\LL(L_1')$ are given by counting strips between even and odd intersections from $\LL$ to $L_i'$ for $i=0,1$. As in \cite{CHL1}, to see the picture more intuitively, we take a small Hamiltonian perturbation $\phi^t$ of $L_i'$, construct Floer complexes $(CF(\LL,\phi^t(L_i')),m_1^t)$ and take the limit $t \to 0$ to get the correct strip counting $m_1: CF(\LL,L_i') \to CF(\LL,L_i').$ According to the definition of the graded localized mirror functor, we fix a character $-j \mapsto e^{\pi i \cdot (-\frac{2j}{3})}$, i.e. $\alpha_{-j}=-\frac{2j}{3}$, where the components are labelled by $L=L^0,$ $\tau(L)=L^{-1}$, $\tau^2(L)=L^{-2}.$ For $L_0=\CS_0(L_{(1,0)})$, $\deg(a_0)=0$, $\deg(a_i)=2$, $\deg(b_0)=1=\deg(b_i)$ for $i=1,2,3$. By Theorem \ref{CTtheorem}, the 0th part $M_0^0$ of the corresponding matrix factorization is $R(0) \oplus R(-1)^3$ where $R(0)$ comes from $a_0$ and $R(-1)^3$ from $a_i$ for $i=1,2,3$.

\begin{figure}
\includegraphics[height=3in]{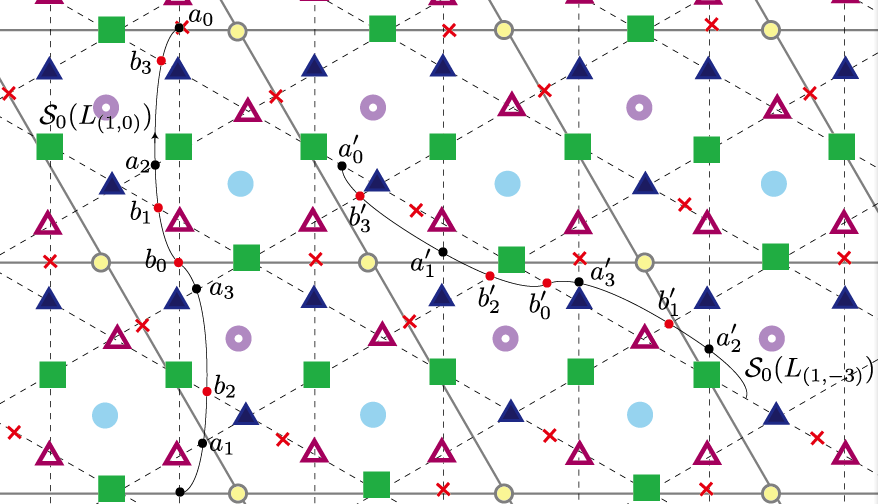}
\caption{(Hamiltonian perturbations of) the images of $L_{(1,0)}$ and $L_{(1,-3)}$ under $\CS_0.$ Black dots are even intersections and red dots are odd intersections.}
\label{4x4mf1}
\end{figure}

Similarly $M^1_0$ which is at the 1st position of $M_0$ is $R \oplus R(1)^3$ and the basis is $\{b_0,b_1,b_2,b_3\}.$ It is also straightforward to see that $M_1^0$ and $M_1^1$ are $R(1) \oplus R^3$ and $R(1) \oplus R(2)^3$ respectively, and their basis are similarly given by $\{a_0',a_1',a_2',a_3'\}$ and $\{b_0',b_1',b_2',b_3'\}$. Under these basis choices, by computations given in Chapter 7 of \cite{CHL1} $M_0$ is as follows:

\[\xymatrixrowsep{4pc}\xymatrixcolsep{2.6pc}\xymatrix{
\cdots\ar[r] &
R(-3) \oplus R(-4)^3  \ar[rr]^-{\tiny\left(\begin{smallmatrix}0 & x & y & z \\x & 0 & \frac{w_z}{\alpha} & -\frac{w_y}{\alpha} 
 \\y & -\frac{w_z}{\alpha} & 0 & \frac{w_x}{\alpha} \\z & \frac{w_y}{\alpha} & -\frac{w_x}{\alpha} & 0\end{smallmatrix}\right)} & & R(-3)\oplus R(-2)^3 \ar[rr]^-{\tiny\left(\begin{smallmatrix} 0 & w_x & w_y & w_z \\ w_x &  0 & -\alpha z & \alpha y \\ w_y & \alpha z & 0 & -\alpha x \\ w_z & -\alpha y & \alpha x & 0\end{smallmatrix}\right)} 
 & &
 R \oplus R(-1)^3 \ar[r]
 &  \cdots}\]
 and $M_1$ is given by:
 
 \[\xymatrixrowsep{4pc}\xymatrixcolsep{2.6pc}\xymatrix{
\cdots\ar[r] &
R(-2) \oplus R(-3)^3  \ar[rr]^-{\tiny\left(\begin{smallmatrix}0 & x & y & z \\x & 0 & \frac{w_z}{\alpha} & -\frac{w_y}{\alpha} 
 \\y & -\frac{w_z}{\alpha} & 0 & \frac{w_x}{\alpha} \\z & \frac{w_y}{\alpha} & -\frac{w_x}{\alpha} & 0\end{smallmatrix}\right)} & & R(-2)\oplus R(-1)^3 \ar[rr]^-{\tiny\left(\begin{smallmatrix} 0 & w_x & w_y & w_z \\ w_x &  0 & -\alpha z & \alpha y \\ w_y & \alpha z & 0 & -\alpha x \\ w_z & -\alpha y & \alpha x & 0\end{smallmatrix}\right)} 
 & &
 R(1) \oplus R^3 \ar[r]
 &  \cdots.}\]

 So we observe that $\CG_0 \circ \CF$ and $\locmir_{gr}^\LL \circ \CS_0$ are identical on objects $L_0$ and $L_1$.

As noticed above, morphisms between $M_0[i]$ and $M_1[j]$ $(i,j=0,1)$ are determined by constant entry parts, so we do not compute all entries of corresponding morphisms. An intersection $(0,0) \in CF(L_0,L_1)$ is mapped to $(0,1/2)$, which correspond to $y \in CF(\phi^t(L_0'),\phi^t(L_1'))$ as in Figure \ref{morphism}. Similarly, $\CS_0(1/3,0)=x$ and $\CS_0(2/3,0)=z$. 

\begin{figure}
\includegraphics[height=3.5in]{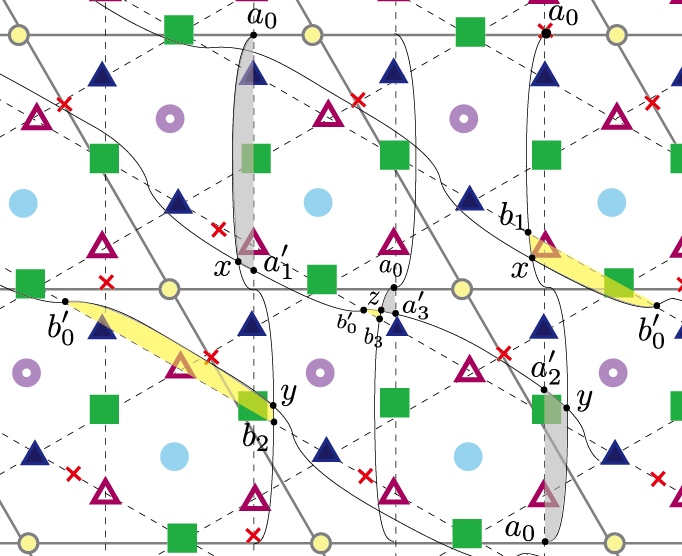}
\caption{$CF(\phi^t(L_0'),\phi^t(L_1'))$ generated by $x$, $y$ and $z$, which map $b_i \mapsto b_0'$(yellow strips) and $a_0 \mapsto a_i' $(gray strips) for $i=1,2,3.$}
\label{morphism}
\end{figure}

 As $t \to 0$, strips in Figure \ref{morphism} collapse to strips of area zero, and they contribute with positive signs by the criterion discussed in section \ref{subsec:fukaya}. It is clear that there are no more strips which contribute to morphisms $b_i \mapsto b_0'$ and $a_0 \mapsto a_i'.$ For example, the morphism of matrix factorizations given by $x$ is given as follows:
 
\[\xymatrixrowsep{4pc}\xymatrixcolsep{2.6pc}\xymatrix{
\cdots\ar[r] &
R(-3) \oplus R(-4)^3 \ar[d]_{p^{-2}}\ar[rr]^-{\tiny\left(\begin{smallmatrix}0 & x & y & z \\x & 0 & \frac{w_z}{\alpha} & -\frac{w_y}{\alpha} 
 \\y & -\frac{w_z}{\alpha} & 0 & \frac{w_x}{\alpha} \\z & \frac{w_y}{\alpha} & -\frac{w_x}{\alpha} & 0\end{smallmatrix}\right)} & & R(-3)\oplus R(-2)^3 \ar[d]_{p^{-1}}\ar[rr]^-{\tiny\left(\begin{smallmatrix} 0 & w_x & w_y & w_z \\ w_x &  0 & -\alpha z & \alpha y \\ w_y & \alpha z & 0 & -\alpha x \\ w_z & -\alpha y & \alpha x & 0\end{smallmatrix}\right)} 
 & &
 R \oplus R(-1)^3 \ar[d]^{p^0}\ar[r]
 &  \cdots \\
\cdots\ar[r] &
R(-2) \oplus R(-3)^3  \ar[rr]^-{\tiny\left(\begin{smallmatrix}0 & x & y & z \\x & 0 & \frac{w_z}{\alpha} & -\frac{w_y}{\alpha} 
 \\y & -\frac{w_z}{\alpha} & 0 & \frac{w_x}{\alpha} \\z & \frac{w_y}{\alpha} & -\frac{w_x}{\alpha} & 0\end{smallmatrix}\right)} & & R(-2)\oplus R(-1)^3 \ar[rr]^-{\tiny\left(\begin{smallmatrix} 0 & w_x & w_y & w_z \\ w_x &  0 & -\alpha z & \alpha y \\ w_y & \alpha z & 0 & -\alpha x \\ w_z & -\alpha y & \alpha x & 0\end{smallmatrix}\right)} 
 & &
 R(1) \oplus R^3 \ar[r]
 &  \cdots}\]
 with $p^{2i}=\left(\begin{array}{cccc} * & * & * & * \\ 1 & * & * & * \\ 0 & * & * & * \\ 0 & * & * & * \end{array}\right)$ and $p^{2i-1}=\left(\begin{array}{cccc} * & -1 & 0 & 0 \\ * & * & * & * \\ * & * & * & * \\ * & * & * & * \end{array}\right).$ 
 
 Similarly, the morphism induced by $y$ is given by 
 \[q^{2i}=\left(\begin{array}{cccc} * & * & * & * \\ 0 & * & * & * \\ 1 & * & * & * \\ 0 & * & * & * \end{array}\right), \; q^{2i-1}=\left(\begin{array}{cccc} * & 0 & -1 & 0 \\ * & * & * & * \\ * & * & * & * \\ * & * & * & * \end{array}\right)\]
 and $z$ gives the morphism
  \[r^{2i}=\left(\begin{array}{cccc} * & * & * & * \\ 0 & * & * & * \\ 0 & * & * & * \\ 1 & * & * & * \end{array}\right), \; r^{2i-1}=\left(\begin{array}{cccc} * & 0 & 0 & -1 \\ * & * & * & * \\ * & * & * & * \\ * & * & * & * \end{array}\right).\]

Therefore, $\locmir_{gr}^\LL\circ \CS_0$ and $\CG_0 \circ \CF$ are identical on morphisms $L_{(1,0)} \to L_{(1,-3)}.$
  
We can also obtain morphisms of graded matrix factorizations $M_1 \to M_0[1]$ from morphisms $L_1' \to L_0'[1]$ in the Fukaya category. As already commented, it suffices to calculate constant entries. They are given by holomorphic strips in Figure \ref{deg1morph}. Again in this case there are no more strips which map $a_0' \mapsto b_i$ and $a_i' \mapsto b_0$. It is also clear that they collapse to area zero and have positive signs.

\begin{figure}
\includegraphics[height=3.5in]{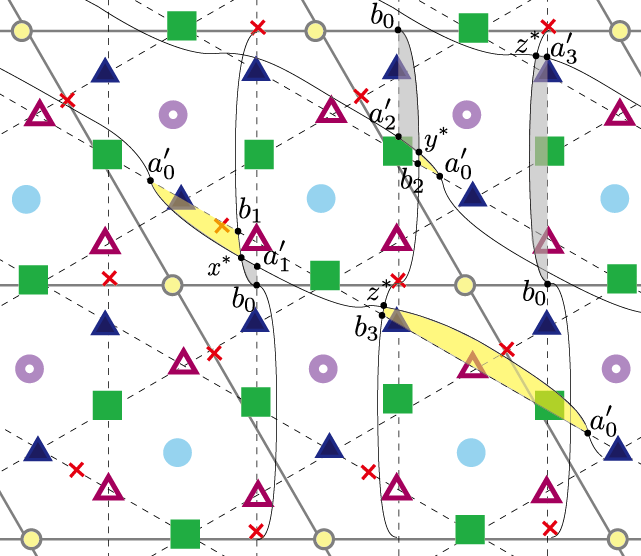}
\caption{$CF(\phi^t(L_1'),\phi^t(L_0'[1]))$ generated by $x^*$, $y^*$ and $z^*$ which are equal to $x$, $y$ and $z$ as intersection points. They map $a_0' \mapsto b_i$(yellow strips), $a_i' \mapsto b_0$(gray strips) for $i=1,2,3$.}
\label{deg1morph}
\end{figure}

Hence, the morphism of graded matrix factorizations induced by $x^*$ is given by

\[\xymatrixrowsep{4pc}\xymatrixcolsep{2.6pc}\xymatrix{
\cdots\ar[r] &
R(-2) \oplus R(-3)^3  \ar[d]_{p'^{-2}}\ar[rr]^-{\tiny\left(\begin{smallmatrix}0 & x & y & z \\x & 0 & \frac{w_z}{\alpha} & -\frac{w_y}{\alpha} \\y & -\frac{w_z}{\alpha} & 0 & \frac{w_x}{\alpha} \\z & \frac{w_y}{\alpha} & -\frac{w_x}{\alpha} & 0\end{smallmatrix}\right)} & & R(-2)\oplus R(-1)^3 \ar[d]_{p'^{-1}}\ar[rr]^-{\tiny\left(\begin{smallmatrix} 0 & w_x & w_y & w_z \\ w_x &  0 & -\alpha z & \alpha y \\ w_y & \alpha z & 0 & -\alpha x \\ w_z & -\alpha y & \alpha x & 0\end{smallmatrix}\right)} 
 & &
 R(1) \oplus R^3 \ar[d]^{p'^0}\ar[r]
 &  \cdots \\
\cdots\ar[r] &
R(-3) \oplus R(-2)^3 \ar[rr]^-{\tiny\left(\begin{smallmatrix} 0 & w_x & w_y & w_z \\ w_x &  0 & -\alpha z & \alpha y \\ w_y & \alpha z & 0 & -\alpha x \\ w_z & -\alpha y & \alpha x & 0\end{smallmatrix}\right)} & & R\oplus R(-1)^3 \ar[rr]^-{\tiny\left(\begin{smallmatrix}0 & x & y & z \\x & 0 & \frac{w_z}{\alpha} & -\frac{w_y}{\alpha} 
 \\y & -\frac{w_z}{\alpha} & 0 & \frac{w_x}{\alpha} \\z & \frac{w_y}{\alpha} & -\frac{w_x}{\alpha} & 0\end{smallmatrix}\right)}
 & &
 R \oplus R(1)^3 \ar[r] &  \cdots}\]
such that $p'^{2i}=\left(\begin{array}{cccc}0 & 1 & 0 & 0 \\1 & * & * & * \\0 & * & * & * \\0 & * & * & *\end{array}\right).$ 

Then 
\begin{equation}\label{pcommuting}\left(\begin{matrix}0 & x & y & z \\x & 0 & \frac{w_z}{\alpha} & -\frac{w_y}{\alpha} 
 \\y & -\frac{w_z}{\alpha} & 0 & \frac{w_x}{\alpha} \\z & \frac{w_y}{\alpha} & -\frac{w_x}{\alpha} & 0\end{matrix}\right)\circ p'^{2i-1}=\left(\begin{array}{cccc}0 & 1 & 0 & 0 \\1 & * & * & * \\0 & * & * & * \\0 & * & * & *\end{array}\right) \circ \left(\begin{matrix} 0 & w_x & w_y & w_z \\ w_x &  0 & -\alpha z & \alpha y \\ w_y & \alpha z & 0 & -\alpha x \\ w_z & -\alpha y & \alpha x & 0\end{matrix}\right).
 \end{equation} On the other hand, recalling (\ref{deg1mfmorph}), the morphism $\{f'^j\}_{j\in \ZZ}$ was induced by the lifting of the map $(1\;\;0\;\;0):R^3 \to R,$ namely 
 \begin{equation*}
 \left(\begin{matrix}0 & x & y & z\end{matrix}\right)\circ f'^{-1}=(1\;\;0\;\;0)\circ \left(\begin{matrix} w_x &  0 & -\alpha z & \alpha y \\ w_y & \alpha z & 0 & -\alpha x \\ w_z & -\alpha y & \alpha x & 0\end{matrix}\right),
 \end{equation*}
$f'^{-2}$ is defined as the successive lifting of $f'^{-1}$ and then the lifting becomes 2-periodic. It is clear that $p'^{-1}$ is also realized by the lift of $(1\;\;0\;\;0):R^3 \to R$, i.e.
\[\left(\begin{matrix}0 & x & y & z\end{matrix}\right)\circ p'^{-1}=(1\;\;0\;\;0)\circ \left(\begin{matrix} w_x &  0 & -\alpha z & \alpha y \\ w_y & \alpha z & 0 & -\alpha x \\ w_z & -\alpha y & \alpha x & 0\end{matrix}\right),\]
so $\{p'^j\}_{j\in \ZZ}$ is the same morphism as that induced in (\ref{deg1mfmorph}). Similarly, $y^*$ induces a morphism \[ q'^{2i}=\left(\begin{array}{cccc}0 & 0 & 1 & 0 \\0 & * & * & * \\1 & * & * & * \\0 &  * & * & *\end{array}\right)\]
and $z^*$ induces
\[ r'^{2i}=\left(\begin{array}{cccc}0 & 0 & 0 & 1 \\0 & * & * & * \\0 & * & * & * \\1 &  * & * & *\end{array}\right)\]
and they also induce same morphisms as those coming from $\CG_0 \circ \CF.$

Finally we construct $\widetilde{\CF}$.
By \cite{CHL1}, the functor $\locmir^\LL_{gr}\circ \CS_0$ is an $\AI$-quasiequivalence between $Fu_0(E)$ and $MF_\ZZ(W)$. By the result of \cite{CT}, the functor $\CG_0$ also extends to an $\AI$-equivalence between $D^b_\infty Coh(X)$ and $MF_\ZZ(W)$. Hence we define an $\AI$-functor $\widetilde{\CF}:Fu_0(E) \to D^b_\infty Coh(X)$ by $\widetilde{\CF}:=\CG_0^{-1} \circ \locmir^\LL_{gr}\circ \CS_0,$ so that it is an $\AI$-quasiequivalence which realizes the Polishchuk-Zaslow's mirror functor $\CF$ on $\calA$. The cohomology functor of $\widetilde{\CF}$ gives an exact functor which gives the CY-CY homological mirror symmetry of $T^2$ of \cite{AS}.

\subsection{Description of $\CS_i$ for any $i\in \ZZ$} 
By definition of $\RR \omega_i$, it is easy to observe the following:
\[ \RR\omega_i(\CO(-i)) \cong \RR\omega_0(\CO)(-i),\;
 \RR\omega_i(\CO(-i+1)) \cong \RR \omega_0(\CO(1))(-i).\]

By $\CF$, $L_{(1,3i)}$ corresponds to $\CO(-i)$, so via $\CG_i \circ \CF$, $L_{(1,3i)}$ corresponds to the matrix factorization $M_0(-i)$ and $L_{(1,3i-3)}$ corresponds to $M_1(-i)$. On the other hand, if $\locmir_{gr}^\LL(L')=M'$, then $\locmir_{gr}^\LL(\tau^{-1}(L'))=M'(-1)$ by definition of $\locmir_{gr}^\LL.$ Recall that $L_0=\CS_0(L_{(1,0)})$ is mapped to $M_0$ via $\locmir_{gr}^\LL$. $\tau^{-3}$ is the rotation by $2\pi$, which corresponds to $[-1]$ in $D^\pi Fu_0(T^2).$ Let $\displaystyle j=\lfloor -\frac{i}{3} \rfloor \big( = \sup_{n \in \ZZ} \{n \leq -\frac{i}{3}\}\big)$ and $-i=3j+d$. Then $d=0$, 1 or 2. Now define a symplectomorphism \[S_i := \tau^{d} \circ t_{(0,1/2)} \circ \left(\begin{array}{cc}1 & 0 \\-3i+2 & 1\end{array}\right). \] The symplectomorphism $t_{(0,1/2)}\circ\left(\begin{array}{cc}1 & 0 \\-3i+2 & 1\end{array}\right)$ maps $L_{(1,3i)}$ to $L_0$ and $L_{(1,3i-3)}$ to $L_1$. When $d\in \{0,1,2\}$, $\tau^d$ can be also considered to be a symplectomorphism. Finally, let \[\CS_i:=[-j]\circ S_i.\] 
\bibliographystyle{amsalpha}

\end{document}